\documentclass[reqno,twoside]{amsart}

\usepackage{amsfonts}
\usepackage{amsmath,amssymb,amsthm,amsthm}
\usepackage{mathtools}
\usepackage{dsfont}
\usepackage{etoolbox}
\usepackage{color}
\usepackage{mathrsfs}

\usepackage[dvips,bottom=1.4in,right=1in,top=1in, left=1in]{geometry}

\makeatletter
\def\eqnarray{\stepcounter{equation}\let\@currentlabel=\theequation
\global\@eqnswtrue
\tabskip\@centering\let\\=\@eqncr
$$\halign to \displaywidth\bgroup\hfil\global\@eqcnt\z@
  $\displaystyle\tabskip\z@{##}$&\global\@eqcnt\@ne
  \hfil$\displaystyle{{}##{}}$\hfil
  &\global\@eqcnt\tw@ $\displaystyle{##}$\hfil
  \tabskip\@centering&\llap{##}\tabskip\z@\cr}

\def\endeqnarray{\@@eqncr\egroup
      \global\advance\c@equation\m@ne$$\global\@ignoretrue}

\newtheorem{theorem}{Theorem}[section]

\newtheorem{definition}[theorem]{Definition}

\newtheorem{lemma}[theorem]{Lemma}

\newtheorem{proposition}[theorem]{Proposition}

\newtheorem{remark}[theorem]{Remark}
\numberwithin{equation}{section}

\def\RR{{\mathbb{R}}}
\def\NN{{\mathbb{N}}}

\def\Om{\Omega}

\def\bOm{\overline{\Om}}
\def\pOm{\partial\Omega}

\newcommand{\norm}[2]{{\left\|#1\right\|}_{#2}}

\newcommand{\fl}[2]{(-\Delta)^#1#2}

\title[Local regularity  for the fractional Laplacian]{Local elliptic regularity for the Dirichlet fractional Laplacian}

\author{Umberto Biccari}

\address{Umberto Biccari, DeustoTech, University of Deusto, 48007 Bilbao, Basque Country, Spain.}
\address{Umberto Biccari, Facultad de Ingenier\'{\i}a, Universidad de Deusto, Avda Universidades 24, 48007 Bilbao, Basque Country, Spain.}
\email{umberto.biccari@deusto.es, u.biccari@gmail.com}

\author{Mahamadi Warma}

\address{Mahamadi Warma, University of Puerto Rico  (Rio Piedras Campus), College of Natural Sciences,
Department of Mathematics, PO Box 70377 San Juan PR
00936-8377 (USA). }
\email{mahamadi.warma1@upr.edu, mjwarma@gmail.com}

\author{Enrique Zuazua}
\address{Enrique Zuazua, DeustoTech, University of Deusto, 48007 Bilbao, Basque Country, Spain.}
\address{Enrique Zuazua, Facultad de Ingenier\'{\i}a, Universidad de Deusto, Avda Universidades 24, 48007 Bilbao, Basque Country, Spain.}
\address{Enrique Zuazua, Departamento de Matem\'aticas, Universidad Aut\'onoma de Madrid, Campus de Cantoblanco,
28049, Madrid, Spain}
\email{enrique.zuazua@deusto.es, enrique.zuazua@uam.es}



\begin{document}

\begin{abstract}
We analyze the local elliptic regularity of weak solutions to the Dirichlet problem associated with the fractional Laplacian $\fl{s}{}$ on an arbitrary bounded open set $\Omega\subset\RR^N$. For $1<p<2$, we obtain regularity in the Besov space  $B^{2s}_{p,2,\textrm{loc}}(\Omega)$, while for $2\leq p<\infty$ we show that the solutions belong to $W^{2s,p}_{\textrm{loc}}(\Omega)$. The key tool consists in analyzing carefully the elliptic equation satisfied by the solution locally, after cut-off, to later employ sharp regularity results  in the whole space. We do it by two different methods. First working directly in the variational formulation of the elliptic problem and then employing the heat kernel representation of solutions.
\end{abstract}

\maketitle

\begin{center}
{\it Dedicated to Ireneo Peral on the occasion of his 70th birthday: Gracias Ireneo por tantos a\~nos de amistad y ejemplo.}
\end{center}

\section{Introduction}\label{intro}
The aim of the present paper is to study the local elliptic regularity of weak solutions to the following Dirichlet problem
\begin{equation}\label{DP}
\begin{cases}
(-\Delta)^su=f\;\;\;&\mbox{ in }\;\Omega\\
u=0&\mbox{ on }\;\RR^N\setminus\Omega,
\end{cases}
\end{equation}
where $\Omega\subset\RR^N$ is an arbitrary bounded open set and $s\in(0,1)$. 

Here $f$ is a given distribution and $(-\Delta)^s$ denotes the fractional Laplace operator, which is defined as the following singular integral 
\begin{align}\label{fl}
(-\Delta)^su(x):=C_{N,s}\,\mbox{P.V.}\int_{\RR^N}\frac{u(x)-u(y)}{|x-y|^{N+2s}}\;dy,\;\;x\in\RR^N.
\end{align}
In \eqref{fl}, $C_{N,s}$ is a normalization constant, given by
\begin{align*}
C_{N,s}:=\frac{s2^{2s}\Gamma\left(\frac{2s+N}{2}\right)}{\pi^{\frac
N2}\Gamma(1-s)},
\end{align*}
$\Gamma $ being the usual Gamma function. Moreover, we have to mention that, for having a completely rigorous definition of the fractional Laplace operator, it is necessary to introduce also the class of functions $u$ for which computing $(-\Delta)^s u$ makes sense. We postpone this discussion to the next section.

Models involving the fractional Laplacian or other types of non-local operators have been recently used in the description of several complex phenomena for which the classical local approach turns up to be inappropriate or limited. Among others, we mention applications in elasticity (\cite{DPV}), turbulence (\cite{TURB}), anomalous transport and diffusion (\cite{BTG,AN_DIFF}), porous media flow (\cite{VAZ}), image processing (\cite{GO}), wave propagation in heterogeneous high contrast media (\cite{WAVE}). Also, it is well known that the fractional Laplacian is the generator of s-stable processes, and it is often used in stochastic models with applications, for instance, in mathematical finance (\cite{LEVI,PHAM}).  

One of the main differences between these non-local models and  classical Partial Differential Equations is  that the fulfilment of a non-local equation  at a point involves the values of the function far away from that point. 

Our concern in this article is the study of the local elliptic regularity for weak solutions of the Dirichlet problem \eqref{DP}. For this purpose, we firstly remind that, according to \cite{LPPS}, we have the following definition of weak solutions.

\begin{definition}\label{weak_sol_def-en}
Let $f\in W^{-s,2}(\bOm)$. A function $u\in W_0^{s,2}(\bOm)$ is said to be a finite energy solution of the Dirichlet problem \eqref{DP} if for every $v\in W_0^{s,2}(\bOm)$, the equality
\begin{align}\label{wek-sol-en}
\frac{C_{N,s}}{2}\int_{\RR^N}\int_{\RR^N}\frac{(u(x)-u(y))(v(x)-v(y))}{|x-y|^{N+2s}}\;dxdy=\langle f,v\rangle_{W^{-s,2}(\bOm),W_0^{s,2}(\bOm)}
\end{align}
holds.
\end{definition}
We notice that, when $1<p<2$, it is not natural to consider finite energy solutions for \eqref{DP}, and we shall rather introduce an alternative notion of solution. This will be given by duality with respect to the following class of test functions:

\begin{align*}
	\mathcal{T}(\Omega) = \Big\{\phi : \fl{s}{\phi} = \psi\;\;\textrm{ in }\;\Omega,\;\phi=0\;\;\textrm{ in }\;\RR^N\setminus\Omega,\;\psi\in C_0^{\infty}(\Omega)\Big\}.
\end{align*}

\begin{definition}\label{weak-sol-def}
Let $1<p<2$. We say that $u\in L^1(\Omega)$ is a weak solution to \eqref{DP} if, for $f\in L^1(\Omega)$ we have that
\begin{align*}
	\int_{\Omega} u\psi\,dx = \int_{\Omega}f\phi\,dx,
\end{align*} 
for any $\phi\in\mathcal{T}(\Omega)$ with $\psi\in C_0^{\infty}(\Omega)$.
\end{definition}

According to the definitions above, if $f\in L^p(\Omega)$, with  $p\ge 2$, finite energy solutions of  \eqref{DP} will be considered while, if $1<p<2$, solutions will be understood in the sense of duality/transposition. In both cases, we shall refer to them as weak solutions. Moreover, we notice that, according to Definition \ref{weak-sol-def}, duality solutions do not require that $f$ belongs to the dual space $W^{-s,2}(\bOm)$. Finally, we also notice that, if $f\in L^p(\Omega)$ with $p\geq 2$, we have the continuous embedding  $L^p(\Omega)\hookrightarrow L^2(\Omega)\hookrightarrow W^{-s,2}(\bOm)$, meaning that the property $f\in W^{-s,2}(\bOm)$ is automatically guaranteed.

\noindent The following  $W_{\rm loc}^{2s,2}(\Omega)$-regularity property is our first main result.

\begin{theorem}[\bf $L^2$-Local regularity]\label{reg-2}
Let $f\in W^{-s,2}(\bOm)$ and let $u\in W_0^{s,2}(\bOm)$ be the unique weak solution to the Dirichlet problem \eqref{DP}.  If $f\in L^2(\Omega)$, then $u\in W_{\rm loc}^{2s,2}(\Omega)$.
\end{theorem}

\noindent This result can be be extended to the $L^p$ setting as follows.
 
\begin{theorem}\label{reg-p}
Let $1<p<\infty$. Given $f\in L^p(\Omega)$, let $u$ be the unique weak solution to the Dirichlet problem \eqref{DP}. Then $u\in \mathscr{L}^p_{2s,\rm loc}(\Omega)$. As a consequence we have the following result.
\begin{enumerate}
\item If $1<p<2$ and $s\ne 1/2$, then $u\in B^{2s}_{p,2,\rm loc}(\Omega)$.

\item If $1<p<2$ and $s= 1/2$, then $u\in W_{\rm loc}^{2s,p}(\Omega)=W_{\rm loc}^{1,p}(\Omega)$.	
	
\item If $2\leq p<\infty$, then $u\in W^{2s,p}_{\rm loc}(\Omega)$.
\end{enumerate}
\end{theorem} 

In Theorems \ref{reg-2} and \ref{reg-p}, $W_0^{s,2}(\bOm)$ denotes the fractional order Sobolev space which consists of all functions $u\in W^{s,2}(\RR^N)$ which are zero on $\RR^N\setminus\Omega$, while $W^{-s,2}(\bOm)$ is its dual. We will give a more exhaustive description of these spaces in Appendix \ref{appendix} at the end of this paper. Moreover, with $\mathscr{L}^p_{2s,\rm loc}(\Omega)$ we indicate the potential space 
\begin{align}\label{sp-stein-loc}
	\mathscr{L}^p_{2s,\rm loc}(\Omega):=\Big\{ u\in L^p(\Omega): u\eta \in \mathscr{L}_{2s}^p(\RR^N) \;\textrm{ for any test function $\eta\in\mathcal{D}(\Omega)$} \Big\}.
\end{align}
Analogously, with $B^{2s}_{p,2,\rm loc}(\Omega)$ we indicate the Besov space
\begin{align}\label{besov-loc}
	B^{2s}_{p,2,\rm loc}(\Omega):=\Big\{ u\in L^p(\Omega): u\eta \in B^{2s}_{p,2}(\RR^N) \;\textrm{ for any test function $\eta\in\mathcal{D}(\Omega)$} \Big\}.
\end{align}

Finally, for the definitions of the Besov and potential spaces $B^{2s}_{p,2}(\RR^N)$ and $\mathscr{L}_{2s}^p(\RR^N)$ we refer again to Appendix \ref{appendix}.

We have to notice that Theorems \ref{reg-2} and \ref{reg-p} are already known when $\Omega$ is the whole space $\RR^N$. In fact they follow by combining several results on Fourier transform and singular integrals contained in \cite[Chapter V]{STEIN}. Moreover, our results complement some  previous ones on local and global Sobolev regularity. 
\begin{itemize}
	\item In \cite[Theorem 17]{LPPS}, Leonori, Peral, Primo and Soria show that, if $f\in L^m(\Omega)$ for some $m\ge \frac{2N}{N+2s}$, then the weak solution $u$ of \eqref{DP} belongs to $W_0^{s\theta,p}(\Omega)$ for some $0<\theta<1$ and $p$ such that 
	\begin{align*}
		\frac{1}{p} = \frac{1}{m}+\theta\left(\frac{1}{2}-\frac{1}{m}\right)-\frac{2s(1-\theta)}{N}.
	\end{align*}	
	This is proved by means of an interpolation argument between $W^{s,2}(\Omega)$ and $L^{\frac{mN}{N-2ms}}(\Omega)$. Note however that this global regularity result does not achieve the   maximal gain of regularity since $0<s\theta<s$.
On the other hand, a well-known example shows that the optimal global regularity fails (for more details, see \cite[Remark 7.2]{RS-ES}). 
	\item In \cite{cozzi}, it is proved that if we take $f\in L^2(\Omega)$, then the corresponding weak solution of \eqref{DP} satisfies $u\in W^{2s-\varepsilon,2}_{\textrm{loc}}(\Omega)$ for all $\varepsilon >0$.
\end{itemize}

We also mention that similar results were obtained using pseudo-differential calculus (see, e.g., \cite[Section 7]{grubb} or \cite[Chapter XI, Theorem 2.5 and Exercise 2.1]{taylor}). In particular, in \cite[Section 7]{grubb}, Grubb proved that, for all $1<p<\infty$, the assumption $f\in W^{\tau,p}(\Omega)$ for some real number $\tau\geq 0$ implies that the corresponding solution $u$ of \eqref{DP} belongs to $\mathscr{L}^p_{\tau+2s,\rm loc}(\Omega)$. 

In this article we will resent an alternative approach to the proof of Theorems \ref{reg-2} and \ref{reg-p}, which complements the pseudo-differential one, using merely classical PDE techniques in the context of linear and nonlinear  elliptic and parabolic equations.  

Finally, we remind that, for the classical Laplace operator, maximum  regularity holds globally provided that the open set is smooth enough.  This result fails for the fractional Laplace operator. We refer to Section \ref{open-pb} for a full discussion on this topic and the possible remedies that should involve weighted estimates to take into account the boundary singularities. 

The strategy that we will employ to prove our local regularity theorems does not involve neither interpolation techniques, as in the proof of \cite[Theorem 17]{LPPS}, nor the theory of pseudo-differential operators, as in \cite{grubb}. Instead, it will be based on a cut-off argument that will allow us to reduce the problem to the whole space case, for which, as we have mentioned above, the result is already known (see for example Theorem \ref{re-r-N} below). 

In order to develop this technique, the following proposition, which provides a formula for the fractional Laplacian of the product of two functions, will be fundamental  (see e.g. \cite{ROS} and the references therein).

\begin{proposition}\label{prop-prod-formu}
Let $u$ and $v$ be such that $(-\Delta)^su$ and $(-\Delta)^sv$ exist  and
\begin{align}\label{exit}
\int_{\RR^N}\frac{|(u(x)-u(y))(v(x)-v(y))|}{|x-y|^{N+2s}}\;dy<\infty.
\end{align}
Then $(-\Delta)^s(uv)$ exists and is given by
\begin{align}\label{prod-formu}
(-\Delta)^s(uv)=u(-\Delta)^sv+v(-\Delta)^su-I_s(u,v),
\end{align}
where
\begin{align}\label{Is-formula}
I_s(u,v)(x):=C_{N,s}\int_{\RR^N}\frac{(u(x)-u(y))(v(x)-v(y))}{|x-y|^{N+2s}}\;dy,\;\;x\in\RR^N.
\end{align}
\end{proposition}

\begin{remark}\label{rem-exit}
{\em We mention that for example  if $u,v\in W^{s,2}(\RR^N)$ with  $(-\Delta)^su, (-\Delta)^sv\in L^2(\RR^N)$, then one has \eqref{exit} and thus formula \eqref{prod-formu} holds for such functions.}
\end{remark}

Formula \eqref{prod-formu}, applied to the product of $u$ with a cut-off function $\eta$, will be the principal tool for transforming our original problem \eqref{DP} to one in the whole $\RR^N$. Then, for the proof of our main results, we will need to carefully analyze the regularity of the remainder term $I_s$. 

This analysis will be developed following two different approaches. In the first one, we will consider the fractional Laplacian $(-\Delta)^s$ as defined in \eqref{fl}. In the second one, we will instead use the equivalent characterization of the fractional Laplace operator through the heat semigroup $(e^{t\Delta})_{t\ge 0}$, given by
\begin{align}\label{fl_heat}
	(-\Delta)^s u:=\frac{1}{\Gamma(-s)}\int_0^{+\infty} \Big(e^{t\Delta} u - u\Big)\frac{dt}{t^{1+s}},
\end{align} (see for instance  \cite[Section 2.1]{stinga} and the references therein). We recall that here, $\Gamma(1-s):=-s\Gamma(-s)$.

We mention that this careful analysis of the regularity of the remainder term had been partially developed already in \cite[Lemma B1]{biccari}, as a technical tool for obtaining the results therein presented. This has been one of the main motivations that led to the development of the present work.    

\noindent Finally, we also mention that our techniques and results extend to the following parabolic problem  
\begin{align}\label{parabolic}
	\begin{cases}
	u_t+\fl{s}{u} = f & \textrm{ in }\Omega\times(0,T),
	\\
	u = 0 & \textrm{ on }(\RR^N\setminus\Omega)\times(0,T),
	\\
	u(\cdot,0)=0 & \textrm{ in }\Omega.
	\end{cases}
\end{align}
In particular, we have:
\begin{theorem}\label{reg_thm_parabolic}
Let $1<p<\infty$. Given $f\in L^p(\Omega\times(0,T))$, let $u$ be the unique weak solution to the parabolic problem \eqref{parabolic}. Then $u\in L^p\big((0,T);\mathscr{L}^p_{2s,\rm loc}(\Omega)\big)$. As a consequence we have the following result.
\begin{enumerate}
\item If $1<p<2$ and $s\ne 1/2$, then $u\in L^p\big((0,T);B^{2s}_{p,2,\rm loc}(\Omega)\big)$.
	
\item If $1<p<2$ and $s=1/2$, then $u\in L^p\big((0,T);W^{2s,p}_{\rm loc}(\Omega)\big)=L^p\big((0,T);W^{1,p}_{\rm loc}(\Omega)\big)$.
	
\item If $2\leq p<\infty$, then $u\in L^p\big((0,T);W^{2s,p}_{\rm loc}(\Omega)\big)$.
\end{enumerate}

\end{theorem}
\noindent We refer to \cite{fl_parabolic} for more details on this topic.

The paper is organized as follows. In Section \ref{preliminaries}, we present some preliminary tools that we shall use in the proof of our main results. In Section \ref{local-reg-sec}, we give the proof of Theorems \ref{reg-2} and \ref{reg-p} using the integral representation of the fractional Laplacian. In Section \ref{fl-heat-sec}, we use the second approach which is based on the representation \eqref{fl_heat}. Finally, in Section \ref{open-pb}, we present some open problems and perspectives that are closely related to our work. 

\section{Preliminaries}\label{preliminaries}
In this Section, we introduce some preliminary result that will be useful for the proof of our main Theorems \ref{reg-2} and \ref{reg-p}.

We start by giving a more rigorous definition of the fractional Laplace operator, as we have anticipated in Section \ref{intro}. Let 
\begin{align*}
\mathcal L_s^{1}(\RR^N):=\left\{u:\RR^N\to\RR\;\mbox{ measurable},\; \int_{\RR^N}\frac{|u(x)|}{(1+|x|)^{N+2s}}\;dx<\infty\right\}.
\end{align*}
For $u\in \mathcal L_s^{1}(\RR^N)$ and $\varepsilon>0$ we set
\begin{align*}
(-\Delta)_\varepsilon^s u(x):= C_{N,s}\int_{\{y\in\RR^N:\;|x-y|>\varepsilon\}}\frac{u(x)-u(y)}{|x-y|^{N+2s}}\;dy,\;\;x\in\RR^N.
\end{align*}

\noindent The {\bf fractional Laplace operator}  $(-\Delta)^s$ is then defined by the following singular integral:
\begin{align}\label{fl_def}
(-\Delta)^su(x)=C_{N,s}\,\mbox{P.V.}\int_{\RR^N}\frac{u(x)-u(y)}{|x-y|^{N+2s}}\;dy=\lim_{\varepsilon\downarrow 0}(-\Delta)_\varepsilon^s u(x),\;\;x\in\RR^N,
\end{align}
provided that the limit exists. 

We notice that if $0<s<1/2$ and $u$ is smooth, for example bounded and Lipschitz continuous on $\RR^N$, then the integral in \eqref{fl_def} is in fact not really singular near $x$ (see e.g. \cite[Remark 3.1]{NPV}). Moreover,  $\mathcal L_s^{1}(\RR^N)$ is the right space for which $ v:=(-\Delta)_\varepsilon^s u$ exists for every $\varepsilon>0$, $v$ being also continuous at the continuity points of  $u$. 

The following result of existence and uniqueness of weak solutions to the Dirichlet problem \eqref{DP} is by now well known (see e.g. \cite[Theorem 12]{LPPS}).

\begin{proposition}\label{prop-ex}
Let $\Omega\subset\RR^N$ be an arbitrary bounded open set and $0<s<1$. Then for every $f\in W^{-s,2}(\bOm)$, the Dirichlet problem \eqref{DP} has a unique finite energy solution $u\in W_0^{s,2}(\bOm)$. In addition, there exists a constant $C>0$ such that
\begin{align}\label{est-sol}
\|u\|_{W_0^{s,2}(\bOm)}\le C\|f\|_{W^{-s,2}(\bOm)}.
\end{align}
\end{proposition}

\begin{proof}
For the sake of completeness we include the proof. We recall that a complete description of the functional setting in which we are working is presented in the Appendix \ref{appendix}.
Moreover, we recall also that, according to Definition \ref{weak_sol_def-en}, a function $u\in W_0^{s,2}(\bOm)$ is said to be a weak solution of the Dirichlet problem \eqref{DP} if for every $v\in W_0^{s,2}(\bOm)$, the equality
\begin{align}\label{wek-sol}
\frac{C_{N,s}}{2}\int_{\RR^N}\int_{\RR^N}\frac{(u(x)-u(y))(v(x)-v(y))}{|x-y|^{N+2s}}\;dxdy=\langle f,v\rangle_{W^{-s,2}(\bOm),W_0^{s,2}(\bOm)}
\end{align}
holds. Hence, given $u,v\in W_0^{s,2}(\bOm)$ let us consider the bilinear form
\begin{align}\label{form}
\mathcal E(u,v)=\frac{C_{N,s}}{2}\int_{\RR^N}\int_{\RR^N}\frac{(u(x)-u(y))(v(x)-v(y))}{|x-y|^{N+2s}}\;dxdy,
\end{align}
which is symmetric, continuous and coercive.  

Thus by the classical Lax-Milgram Theorem, for every $f\in (W_0^{s,2}(\bOm))^\star=:W^{-s,2}(\bOm)$, there exists a unique $u\in W_0^{s,2}(\bOm)$ such that the equality \eqref{wek-sol} holds for every $v\in W_0^{s,2}(\bOm)$. We have shown that \eqref{DP} has a unique weak solution $u\in W_0^{s,2}(\bOm)$. Taking $v=u$ as a test function in \eqref{wek-sol} and using \eqref{equi-nrom} we get that
\begin{align*}
C\|u\|_{W_0^{s,2}(\bOm)}^2=\langle f,u\rangle_{W^{-s,2}(\bOm),W_0^{s,2}(\bOm)}\le \|f\|_{W^{-s,2}(\bOm)}\|u\|_{W_0^{s,2}(\bOm)}.
\end{align*}
We have shown \eqref{est-sol} and the proof is finished.
\end{proof}

\begin{remark}
Notice that also for $1<p<2$ existence and uniqueness of a weak solution to problem \eqref{DP} are guaranteed by \cite[Theorem 23]{LPPS}.	
\end{remark}

\begin{remark}
Notice that \cite[Theorem 12]{LPPS} holds for a more general non-local operator where the kernel $|x-y|^{-N-2s}$ is replaced by a general symmetric kernel $K(x,y)$ satisfying $\lambda\le K(x,y)|x-y|^{N+2s}\le \lambda^{-1}$ for all $(x,y)\in\RR^N\times\RR^N$, $x\ne y$,  and for some constant $0<\lambda\le 1$.
\end{remark}

\begin{remark}
Let $f\in W^{-s,2}(\bOm)$ and let $u\in W_0^{s,2}(\bOm)$ be the weak solution of the Dirichlet problem \eqref{DP}. We notice that it follows from the Sobolev embeddings \eqref{sob-emb2} and \eqref{sob-emb1} that if $N<2s$, then $u\in C^{0,s-\frac N2}(\bOm)$ and if $N=2s$, then $u\in L^q(\Omega)$ for every $1\le q<\infty$.
\end{remark}

The following Lemma, giving a precise $L^q$-regularity of weak solutions and complementing the results in \cite[Theorem 16]{LPPS}, will be useful in the sequel.

\begin{lemma}\label{lem-reg-p}
Assume that $N>2s$ and let $f\in L^p(\Omega)$ for some $p\ge \frac{2N}{N+2s}$. Then \eqref{DP} has a unique weak solution $u$. In addition the following assertions hold.
\begin{enumerate}
\item  If $p>\frac{N}{2s}$, then $u\in L^\infty(\Omega)$ and there exists a constant $C>0$ such that
\begin{align*}
\|u\|_{L^\infty(\Omega)}\le C\|f\|_{L^p(\Omega)}.
\end{align*}
\item If  $\frac{2N}{N+2s}\le p\le \frac{N}{2s}$, then $u\in L^q(\Omega)$ for every $q$ satisfying $p\le q<\frac{Np}{N-2sp}$ and there exists a constant $C>0$ such that
\begin{align*}
\|u\|_{L^q(\Omega)}\le C\|f\|_{L^p(\Omega)}.
\end{align*}
\end{enumerate}
\end{lemma}

\begin{proof}
Let $f\in L^p(\Omega)$ for some $p\ge \frac{2N}{N+2s}$. Since $W_0^{s,2}(\bOm)\hookrightarrow L^{\frac{2N}{N-2s}}(\Omega)$ (see \eqref{sob-emb1}), we have that $L^p(\Omega)\hookrightarrow L^{\frac{2N}{N+2s}}(\Omega)\hookrightarrow W^{-s,2}(\bOm)$. Thus \eqref{DP} has a unique weak solution $u$.

Let $\mathcal E$ with domain $D(\mathcal E)=W_0^{s,2}(\bOm)$ be the bilinear, symmetric,  continuous and coercive form defined in \eqref{form}. As we have shown in the proof of Proposition \ref{prop-ex}, for every $f\in W^{-s,2}(\bOm)$  there exists a unique $u\in W_0^{s,2}(\bOm)$ such that
\begin{align*}
\mathcal E(u,v)=\langle f,v\rangle_{W^{-s,2}(\bOm),W_0^{s,2}(\bOm)}, \quad \forall v\in W_0^{s,2}(\bOm).
\end{align*}

This defines an operator $\mathcal A: W_0^{s,2}(\bOm)\to W^{-s,2}(\bOm)$ which is continuous and coercive. Let $A_D$ be the part of $\mathcal A$ in $L^2(\Omega)$, in the sense that
\begin{align*}
D(A_D):=\Big\{u\in W_0^{s,2}(\bOm):\;\mathcal Au\in L^2(\Omega)\Big\},\;A_Du=\mathcal Au.
\end{align*}
Using an integration by parts argument one can show that $A_D$ is given precisely by
\begin{align*}
D(A_D)=\Big\{u\in W_0^{s,2}(\bOm):\; (-\Delta)^su\in L^2(\Omega)\Big\},\;\;A_D=(-\Delta)^su.
\end{align*}

Then $A_D$ is the realization in $L^2(\Omega)$ of the operator $(-\Delta)^s$ with the Dirichlet boundary condition $u=0$ on $\RR^N\setminus\Omega$. 
The operator $A_D$ has a compact resolvent (this follows from the compactness of the embedding from $W_0^{s,2}(\bOm)$ into $L^2(\Omega)$, see \eqref{sob-emb1}) and its first eigenvalue $\lambda_1>0$. 

In addition  $-A_D$ generates a submarkovian strongly continuous semigroup $(e^{-tA_D})_{t\ge 0}$ which is also ultracontractive in the sense that the semigroup maps $L^r(\Omega)$ into $L^m(\Omega)$ for every $t>0$ and $1\le r\le m\le\infty$. More precisely, following line by line the proof of \cite[Theorem 2.16]{GW-CPDE} by using the appropriate estimates, we get that, for every $1\le r\le m\le\infty$ there exists a constant $C>0$ such that for every $f\in L^r(\Omega)$ and $t>0$,

\begin{align}\label{ultra-cont}
\|e^{-tA_D}f\|_{L^m(\Omega)}\le Ce^{-\lambda_1\left(\frac 1r-\frac 1m\right)}t^{-\frac{N}{2s}\left(\frac 1r-\frac 1m\right)}\|f\|_{L^r(\Omega)}.
\end{align}

Since the operator $A_D$ is invertible, it follows from the abstract result in \cite[Theorem 1.10, p.55]{EN} that for every $f\in L^1(\Omega)\cap W^{-s,2}(\bOm)$, the unique solution $u$ of the Dirichlet problem \eqref{DP} is given by
\begin{align*}
u=A_D^{-1}f=\int_0^\infty e^{-tA_D}f\;dt.
\end{align*}

(a) Assume that $p>\frac{N}{2s}$. Then applying \eqref{ultra-cont} with $r=p$ and $m=\infty$, we get that
\begin{align*}
\|u\|_{L^\infty\Omega)}\le& \int_0^\infty\|e^{-tA_D}f\|_{L^\infty(\Omega)}\;dt\le C\int_0^\infty e^{-\frac{\lambda_1}{p}}t^{-\frac{N}{2sp}}\|f\|_{L^p(\Omega)}\;dt\\
= &C\left(\int_1^\infty e^{-\frac{\lambda_1}{p}}t^{-\frac{N}{2sp}}\;dt+\int_0^1 e^{-\frac{\lambda_1}{p}}t^{-\frac{N}{2sp}}\;dt\right)\|f\|_{L^p(\Omega)}.
\end{align*}

The first integral in the right hand side of the previous estimate is always finite. The second integral will be finite if $1-\frac{N}{2sp}>0$. This is equivalent to $p>\frac{N}{2s}$ and we have shown part (a).

(b) Finally assume that $\frac{2N}{N+2s}\le p\le \frac{N}{2s}$ and let $p\le q<\frac{Np}{N-2sp}$.
Then applying \eqref{ultra-cont} with $r=p$ and $m=q$ we get that
\begin{align*}
\|u\|_{L^q(\Omega)}\le& \int_0^\infty\|e^{-tA_D}f\|_{L^q(\Omega)}\;dt\le C\int_0^\infty e^{-\lambda_1\left(\frac 1p-\frac 1q\right)}t^{-\frac{N}{2s}\left(\frac 1p-\frac 1q\right)}\|f\|_{L^r(\Omega)}\;dt\\
= &C\left(\int_1^\infty e^{-\lambda_1\left(\frac 1p-\frac 1q\right)}t^{-\frac{N}{2s}\left(\frac 1p-\frac 1q \right)}dt+\int_0^1 e^{-\lambda_1\left(\frac 1p-\frac 1q\right)}t^{-\frac{N}{2s}\left(\frac 1p-\frac 1q\right)}\;dt\right)\|f\|_{L^p(\Omega)}.
\end{align*}

As above, the first integral is always finite and the second integral will be finite if $1-\frac{N}{2s}\left(\frac 1p-\frac 1q\right)>0$. This is equivalent to $q<\frac{Np}{N-2sp}$. We have shown part (b) and the proof is finished.
\end{proof}

\begin{remark}
Assertion (b) has been also proved in \cite[Theorem 16]{LPPS}. There, the authors obtained the result  adapting  Moser's method in \cite{MOSER}, which allows to obtain the $L^q(\Omega)$ regularity of the solution $u$ to \eqref{DP} employing functions depending nonlinearly on the solution.   To the best of our knowledge, our approach to the proof of Lemma \ref{lem-reg-p}(b) is new.
\end{remark}

\noindent In our discussion, the following result of regularity on the whole space $\RR^N$ will play an important role.

\begin{theorem}\label{re-r-N}
Let $1<p<\infty$. Given $F\in L^p(\RR^N)$, let $u$ be the unique weak solution to the fractional Poisson type equation
\begin{equation}\label{PE}
	\fl{s}{u}=F\;\;\mbox{ in }\;\RR^N.
\end{equation}
Then $u\in \mathscr{L}^p_{2s}(\RR^N)$. As a consequence we have the following. 
\begin{enumerate}
\item If $1<p<2$ and $s\ne 1/2$, then $u\in B^{2s}_{p,2}(\RR^N)$.
\item If $1<p<2$ and $s= 1/2$, then $u\in W^{2s,p}(\RR^N)=W^{1,p}(\RR^N)$.
\item If $2\leq p<\infty$, then $u\in W^{2s,p}(\RR^N)$.
\end{enumerate}
\end{theorem}

Theorem \ref{re-r-N} is a classical result whose proof can be done by combining several results on singular integrals and Fourier transform contained in \cite[Chapter V]{STEIN}. In particular:

\begin{itemize}
\item If $1<p<2$ and $s\ne 1/2$, then the result follows from \cite[Chapter V, Section 5.3, Theorem 5(B)]{STEIN}, which provides the inclusion $\mathscr{L}^p_{2s}(\RR^N)\subset B_{2s}^{p,2}(\RR^N)$. Moreover, an explicit counterexample showing that sharper inclusions are not possible has been given in  \cite[Chapter V, Section 6.8]{STEIN}. 

\item If $1<p<2$ and $s= 1/2$, then  applying \cite[Chapter V, Section 3.3, Theorem 3]{STEIN} we have $\mathscr{L}^p_{2s}(\RR^N)=\mathscr{L}^p_{1}(\RR^N)=W^{1,p}(\RR^N)$.

\item If $2\leq p<\infty$, then \cite[Chapter V, Section 5.3, Theorem 5(A)]{STEIN} yields $u\in B_{2s}^{p,p}(\RR^N)$ and this latter space, by definition, coincides with $W^{2s,p}(\RR^N)$  (see, e.g., \cite[Chapter V, Section 5.1, Formula (60)]{STEIN}).
\end{itemize}

\section{Proof of Theorems \ref{reg-2} and \ref{reg-p}: first approach}\label{local-reg-sec}
\subsection{Proof of the $L^2$-local regularity theorem}

\begin{proof}[Proof of Theorem \ref{reg-2}]

As we have mentioned above, our strategy is based on a cut-off argument that will allow us to show that the solutions of the fractional Dirichlet problem in $\Omega$, after cut-off, are solutions of the elliptic problem on the whole space $\RR^N$, for which Theorem \ref{re-r-N} holds. For this purpose, given $\omega$ and $\widetilde\omega$ two open subsets of the domain $\Omega$ such that $\widetilde\omega\Subset\omega\Subset\Omega$, we introduce a cut-off function $\eta\in \mathcal D(\omega)$ such that
\begin{equation}\label{eta}
\begin{cases}
\eta(x)\equiv 1\;\;\;&\mbox{ if }\; x\in\widetilde\omega\\
0\le \eta(x)\le 1&\mbox{ if }\; x\in\omega\setminus\widetilde\omega\\
\eta(x)=0&\mbox{ if }\; x\in\RR^N\setminus\omega.
\end{cases}
\end{equation}

\noindent Let $f\in W^{-s,2}(\bOm)$ and let $u\in W_0^{s,2}(\bOm)$ be the unique weak solution to the Dirichlet problem \eqref{DP}.  

Let $\omega$ and $\eta\in\mathcal D(\omega)$ be respectively the set and the cut-off function constructed in \eqref{eta}. We consider the function $u\eta$. It is clear that $u\eta\in W^{s,2}(\RR^N)$. It follows from Proposition \ref{prop-prod-formu} and Remark \ref{rem-exit} that
\begin{align}\label{for-del-prod}
(-\Delta)^s(u\eta)-\eta(-\Delta)^su=u(-\Delta)^s\eta  -I_s(u,\eta).
\end{align}
Let
\begin{align*}
g:=u(-\Delta)^s\eta  -I_s(u,\eta).
\end{align*}
We claim that $g\in L^2(\RR^N)$. In fact there exists a constant  $C>0$, independent of $u$,  such that
\begin{align}\label{norm-f}
\|g\|_{L^2(\RR^N)}\le C\|u\|_{W_0^{s,2}(\bOm)}.
\end{align}
Since $u=0$ on $\RR^N\setminus\Omega$ and $(-\Delta)^s\eta\in L^\infty(\RR^N)$, we have that
\begin{align}\label{est1}
\|u(-\Delta)^s\eta\|_{L^2(\RR^N)}^2=\int_{\Omega}|u(-\Delta)^s\eta|^2\;dx\le \|(-\Delta)^s\eta\|_{L^\infty(\Omega)}^2\|u\|_{L^2(\Omega)}^2.
\end{align}
Now, recall from \eqref{Is-formula} that for a.e. $x\in\RR^N$,
\begin{align*}
I_s(u,\eta)(x):=&C_{N,s}\int_{\RR^N}\frac{(u(x)-u(y))(\eta(x)-\eta(y))}{|x-y|^{N+2s}}\;dy\\
=&C_{N,s}\int_{\Omega}\frac{(u(x)-u(y))(\eta(x)-\eta(y))}{|x-y|^{N+2s}}\;dy\\
&+C_{N,s}\eta(x)\int_{\RR^N\setminus\Omega}\frac{u(x)-u(y)}{|x-y|^{N+2s}}\;dy =\mathbb I_1(x)+\mathbb I_2(x),
\end{align*}
where we have set
\begin{align*}
\mathbb I_1(x):=C_{N,s}\int_{\Omega}\frac{(u(x)-u(y))(\eta(x)-\eta(y))}{|x-y|^{N+2s}}\;dy,\;\;x\in\RR^N,
\end{align*}
and
\begin{align*}
\mathbb I_2(x):=C_{N,s}\eta(x)\int_{\RR^N\setminus\Omega}\frac{u(x)-u(y)}{|x-y|^{N+2s}}\;dy,\;\;x\in\RR^N.
\end{align*}
Let us start to estimate the term $\mathbb I_1(x)$. Using the Cauchy-Schwarz inequality, we get that
\begin{align}\label{CS1}
|\mathbb I_1(x)|\le C_{N,s}\left(\int_{\Omega}\frac{|u(x)-u(y)|^2}{|x-y|^{N+2s}}\;dy\right)^{\frac 12}\left(\int_{\Omega}\frac{(|\eta(x)-\eta(y)|^2}{|x-y|^{N+2s}}\;dy\right)^{\frac 12}.
\end{align}

Let $x\in\Omega$ be fixed and $R>0$ such that $\Omega\subset B(x,R)$. Since $\eta$ is a smooth function (in particular Lipschitz continuous on $\RR^N$), we have that there exists a constant $C>0$ (depending on $\eta$) such that
\begin{align*}
\int_{\Omega}\frac{|\eta(x)-\eta(y)|^2}{|x-y|^{N+2s}}\;dy\le C\int_{\Omega}\frac{dy}{|x-y|^{N+2s-2}}\le C\int_{B(x,R)}\frac{dy}{|x-y|^{N+2s-2}}\le C.
\end{align*}
Using the preceding estimate and \eqref{CS1} we get that
\begin{align}\label{I1}
\int_{\RR^N}|\mathbb I_1(x)|^2\;dx\le& C\int_{\RR^N} \int_{\Omega}\frac{|u(x)-u(y)|^2}{|x-y|^{N+2s}}\;dydx
\le  C\|u\|_{W_0^{s,2}(\bOm)}^2.
\end{align}

Concerning the term $\mathbb I_2$, we notice that $\mathbb I_2=0$ on $\RR^N\setminus\omega$. 
In addition, using the Cauchy-Schwarz inequality we get that
\begin{align}\label{E1}
|\mathbb I_2(x)|^2&\le C_{N,s}^2\int_{\RR^N\setminus\Omega}\frac{\eta^2(x)dy}{|x-y|^{N+2s}}\int_{\RR^N\setminus\Omega}\frac{|u(x)-u(y)|^2}{|x-y|^{N+2s}}\;dy.
\end{align}
For any $y\in \RR^N\setminus\Omega$, we have that
\begin{align*}
\frac{\eta^2(x)}{|x-y|^{N+2s}}=\frac{\chi_{\overline{\omega}}(x)\eta^2(x)}{|x-y|^{N+2s}}\le \chi_{\overline{\omega}}(x)\eta^2(x)\sup_{x\in\overline{\omega}}\frac{1}{|x-y|^{N+2s}}.
\end{align*}
Thus there exists a constant $C>0$ such that
\begin{align}\label{E2}
\int_{\RR^N\setminus\Omega}\frac{\eta^2(x)dy}{|x-y|^{N+2s}}\le \chi_{\overline{\omega}}(x)\eta^2(x)\int_{\RR^N\setminus\Omega}\frac{dy}{\mbox{dist}(y,\partial\overline{\omega})^{N+2s}}\le C\chi_{\overline{\omega}}(x)\eta^2(x),
\end{align}
where we have used that the integral is finite which follows from the facts that $\mbox{dist}(\pOm,\partial\overline{\omega})\ge\delta>0$, that the distance function $\mbox{dist}(y,\partial\overline{\omega})$ grows linearly as $y$ tends to infinity and that $N+2s>N$.  

\noindent Since $\chi_{\overline{\omega}}\eta^2\in L^\infty(\omega)$, and using \eqref{E1} and \eqref{E2}, we get that there exists a constant $C>0$ such that
\begin{align}\label{I2}
\int_{\RR^N}|\mathbb I_2(x)|^2\;dx=\int_{\omega}|\mathbb I_2(x)|^2\;dx\le C\int_{\omega}\int_{\RR^N\setminus\Omega}\frac{|u(x)-u(y)|^2}{|x-y|^{N+2s}}\;dydy
\le &C\|u\|_{W_0^{s,2}(\bOm)}^2.
\end{align}

Now estimate \eqref{norm-f} follows from \eqref{est1}, \eqref{I1}, \eqref{I2} and the claim is proved. We have shown that $\eta u$ is a weak solution to the Poisson Equation \eqref{PE} with $F$ given by $F=\eta(-\Delta)^su+g\in L^2(\RR^N)$. It follows from Theorem \ref{re-r-N} that $(\eta u)\in W^{2s,2}(\RR^N)$. Thus $u\in W_{\rm loc}^{2s,2}(\Omega)$ and the proof is complete.
\end{proof}

\subsection{Proof of the $L^p$-local regularity theorem}
We will now use Theorem \ref{reg-2} to prove our local regularity result in the general $L^p$ setting.

\begin{proof}[Proof of Theorem \ref{reg-p}]
We start by noticing that, assuming $f\in L^p(\Omega)\cap W^{-s,2}(\bOm)$, we have that \eqref{DP} has a unique weak solution $u\in W_0^{s,2}(\bOm)$. We divide the proof into two steps. 
\subsubsection*{Step 1. $1<p<2$} If $1<p<2$, then, according to Theorem \ref{theo-sob-emb}, $u\in W^{s,p}(\Omega)$. In particular, $u\in L^p(\Omega)$. 

Let $\omega$ and $\eta\in\mathcal D(\omega)$ be respectively the set and the cut-off function constructed in \eqref{eta}. We consider the function $u\eta\in W^{s,p}(\RR^N)$. As in the proof of Theorem \ref{reg-2},  we have that $(-\Delta)^s(u\eta)$ is given by 
\begin{align*}
	\fl{s}{u\eta} = \eta f + u\fl{s}{\eta} - I_s(u,\eta),
\end{align*}
where the term $I_s(u,\eta)$ has been introduced in \eqref{Is-formula}. Let $\omega_1, \omega_2$ be open sets such that
\begin{align}\label{omega12}
\overline{\omega}\subset\omega_1\subset\overline{\omega}_1\subset\omega_2
\subset\overline{\omega}_2\subset\Omega. 
\end{align} 

Since the function $\eta$ and the set $\omega$ in \eqref{eta} are arbitrary, it follows that  $u\in W^{s,p}(\omega_2)$. Thus we have $u\in W^{s,p}(\omega_2)\cap L^p(\Omega)$. Let
\begin{align*}
g:=u(-\Delta)^s\eta  -I_s(u,\eta).
\end{align*}
We now claim that $g\in L^p(\RR^N)$ and there exists a constant $C>0$ such that
\begin{align}\label{norm-f-p}
\|g\|_{L^p(\RR^N)}\le C\left(\|u\|_{W^{s,p}(\omega_2)}+\|u\|_{L^p(\Omega)}\right).
\end{align}
Indeed, it is clear that $g$ is defined on all $\RR^N$. Moreover
\begin{align}\label{est1-p}
\|u(-\Delta)^s\eta\|_{L^p(\RR^N)}^p=\int_{\Omega}|u(-\Delta)^s\eta|^p\;dx\le \|(-\Delta)^s\eta\|_{L^\infty(\Omega)}^p\|u\|_{L^p(\Omega)}^p.
\end{align}
For estimating the term $I_s$, we use the decomposition
\begin{align*}
I_s(u,\eta)(x):=&C_{N,s}\int_{\RR^N}\frac{(u(x)-u(y))(\eta(x)-\eta(y))}{|x-y|^{N+2s}}\;dy\\
=&C_{N,s}\int_{\omega_1}\frac{(u(x)-u(y))(\eta(x)-\eta(y))}{|x-y|^{N+2s}}\;dy\\
&+C_{N,s}\eta(x)\int_{\RR^N\setminus\omega_1}\frac{u(x)-u(y)}{|x-y|^{N+2s}}\;dy =\mathbb I_1(x)+\mathbb I_2(x),\;\;x\in\RR^N,
\end{align*}
where we have set
\begin{align*}
\mathbb I_1(x):=C_{N,s}\int_{\omega_1}\frac{(u(x)-u(y))(\eta(x)-\eta(y))}{|x-y|^{N+2s}}\;dy,\;\;x\in\RR^N,
\end{align*}
and
\begin{align*}
\mathbb I_2(x):=C_{N,s}\eta(x)\int_{\RR^N\setminus\omega_1}\frac{u(x)-u(y)}{|x-y|^{N+2s}}\;dy,\;\;x\in\RR^N.
\end{align*}
Let $p':=\frac{p}{p-1}$. Using the H\"older inequality, we get that for a.e. $x\in\RR^N$,
\begin{align}\label{CS1-p}
|\mathbb I_1(x)|\le C_{N,s}\left(\int_{\omega_1}\frac{|u(x)-u(y)|^p}{|x-y|^{N+sp}}\;dy\right)^{\frac 1p}\left(\int_{\omega_1}\frac{|\eta(x)-\eta(y)|^{p'}}{|x-y|^{N+sp'}}\;dy\right)^{\frac 1{p'}}.
\end{align}

Let $x\in\omega_1$ be fixed and $R>0$ such that $\omega_1\subset B(x,R)$. Using the Lipschitz continuity of the function $\eta$, we obtain that there exists constant $C>0$ such that
\begin{align}\label{CS2-p}
\int_{\omega_1}\frac{|\eta(x)-\eta(y)|^{p'}}{|x-y|^{N+sp'}}\;dy\le C\int_{\omega_1}\frac{dy}{|x-y|^{N+sp'-p'}}\le C\int_{B(x,R)}\frac{dy}{|x-y|^{N+sp'-p'}}\le C.
\end{align}
Now, using \eqref{CS1-p}, \eqref{CS2-p} and \eqref{ine-dist}, we get 
\begin{align}\label{I1-p}
\int_{\RR^N}|\mathbb I_1(x)|^p\;dx\le& C\left(\int_{\omega_2} \int_{\omega_1}\frac{|u(x)-u(y)|^p}{|x-y|^{N+sp}}\;dydx+\int_{\RR^N\setminus\omega_2} \int_{\omega_1}\frac{|u(x)-u(y)|^p}{|x-y|^{N+sp}}\;dydx\right)\notag\\
\le &C\left(\|u\|_{W^{s,p}(\omega_2)}^p+\int_{\RR^N\setminus\omega_2} \int_{\omega_1}\frac{|u(x)|^p+|u(y)|^p}{(1+|x|)^{N+sp}}\;dydx\right)\notag\\
\le &C\left(\|u\|_{W^{s,p}(\omega_2)}^p+\|u\|_{L^p(\Omega)}^p\right) ,
\end{align}
where we have also used that $u=0$ on $\RR^N\setminus\Omega$.
Recall that $\mathbb I_2=0$ on $\RR^N\setminus\omega$. 
Then using the H\"older inequality, we get that 
\begin{align}\label{E1-1}
|\mathbb I_2(x)|^p\le C\left(\int_{\RR^N\setminus\omega_1}\frac{\eta^{p'}(x)dy}{|x-y|^{N+sp'}}\right)^{p-1}\int_{\RR^N\setminus\omega_1}\frac{|u(x)-u(y)|^p}{|x-y|^{N+sp}}\;dy.
\end{align}
For any $y\in \RR^N\setminus\omega_1$, we have that
\begin{align*}
\frac{\eta^{p'}(x)}{|x-y|^{N+sp'}}=\frac{\chi_{\overline{\omega}}(x)\eta^{p'}(x)}{|x-y|^{N+sp'}}\le \chi_{\overline{\omega}}(x)\eta^{p'}(x)\sup_{x\in\overline{\omega}}\frac{1}{|x-y|^{N+sp'}}.
\end{align*}
So there exists a constant $C>0$ such that
\begin{align}\label{E2-2}
\int_{\RR^N\setminus\omega_1}\frac{\eta^{p'}(x)dy}{|x-y|^{N+sp'}}\le \chi_{\overline{\omega}}(x)\eta^{p'}(x)\int_{\RR^N\setminus\omega_1}\frac{dy}{\mbox{dist}(y,\partial\overline{\omega})^{N+sp'}}\le C\chi_{\overline{\omega}}(x)\eta^{p'}(x).
\end{align}

In \eqref{E2-2} we have also used that the integral is finite which follows from the fact that $\mbox{dist}(\partial\omega_1,\partial\overline{\omega})\ge\delta>0$ together with the fact that $\mbox{dist}(y,\partial\overline{\omega})$ grows linearly as $y$ tends to infinity and $N+sp'>N$.

Since $\chi_{\overline{\omega}}\eta^{p'}\in L^\infty(\omega)$, and using \eqref{E1-1}, \eqref{E2-2} and \eqref{ine-dist}, we also get that there exists a constant $C>0$ such that
\begin{align}\label{I2-p}
	\int_{\RR^N}|\mathbb I_2(x)|^p\;dx=&\int_{\omega}|\mathbb I_2(x)|^p\;dx\le C\int_{\omega}\int_{\RR^N\setminus\omega_1}\frac{|u(x)-u(y)|^p}{|x-y|^{N+sp}}\;dydx\notag
	\\
	\le &C\int_{\omega}\int_{\RR^N\setminus\omega_1}\frac{|u(x)|^p+|u(y)|^p}{(1+|y|)^{N+sp}}\;dydx\le C\|u\|_{L^p(\Omega)}^p,
\end{align}
where we have used again that $u=0$ on $\RR^N\setminus\Omega$.
Estimate \eqref{norm-f-p} follows from \eqref{est1-p}, \eqref{I1-p}, \eqref{I2-p} and we have shown the claim. We therefore proved that $\eta u$ is a weak solution to the Poisson equation \eqref{PE} with $F$ given by $F=\eta f+g$. Since $F\in L^p(\RR^N)$, it follows from Theorem \ref{re-r-N} that $\eta u\in \mathscr{L}^p_{2s}(\RR^N)$. We have shown that $u\in \mathscr{L}^p_{2s,\rm loc}(\Omega)$. As a consequence we have the following results. 
\begin{enumerate}
\item If $s\ne 1/2$, then $\eta u\in B^{2s}_{p,2}(\RR^N)$, hence $u\in B^{2s}_{p,2,\rm loc}(\Omega)$.
\item If $s= 1/2$, then $\eta u\in W^{2s,p}(\RR^N)=W^{1,p}(\RR^N)$, hence $u\in W^{2s,p}_{\textrm{loc}}(\Omega)=W^{1,p}_{\textrm{loc}}(\Omega)$.
\end{enumerate}

\noindent The proof for $1<p<2$ is concluded.

\subsubsection*{Step 2. $p\geq 2$}
Let $f\in W^{-s,2}(\bOm)$ and let $u\in W_0^{s,2}(\bOm)$ be the weak solution to the Dirichlet problem \eqref{DP}.  Let $\omega$ and $\eta\in\mathcal D(\omega)$ be respectively the set and the cut-off function constructed in \eqref{eta}. We consider the function $u\eta\in W^{s,2}(\RR^N)$. Assume that $f\in L^p(\Omega)$ with $p\ge 2$. As in the proof of Theorem \ref{reg-2},  we have that $(-\Delta)^s(u\eta)$ is given by \eqref{for-del-prod}.
Since by assumption $f\in L^p(\Omega)\hookrightarrow L^2(\Omega)$, it follows from Theorem \ref{reg-2} that $u\eta\in W^{2s,2}(\RR^N)$. 

(a) Applying Theorem \ref{theo-sob-emb}(a) with $r=2s$ and $p=2$, we get that $W^{2s,2}(\RR^N)\hookrightarrow W^{s,\frac{2N}{N-2s}}(\RR^N)$. We have shown that $u\eta\in W^{s,\frac{2N}{N-2s}}(\RR^N)$. Let $\omega_1, \omega_2$ be open sets such that
\begin{align}\label{omega12}
\overline{\omega}\subset\omega_1\subset\overline{\omega}_1\subset\omega_2
\subset\overline{\omega}_2\subset\Omega. 
\end{align} 

Since the function $\eta$ and the set $\omega$ in \eqref{eta} are arbitrary, it follows from the observation $u\eta\in W^{s,\frac{2N}{N-2s}}(\RR^N)$ that  $u\in W^{s,\frac{2N}{N-2s}}(\omega_2)$.
Let $q:=\min\{p,\frac{2N}{N-2s}\}$. We notice that $q\geq 2$. Applying again Theorem \ref{theo-sob-emb}(a) with $r=2s$ and $p=2$ and the above $q$, we also get that $W^{2s,2}(\omega_2)\hookrightarrow W^{s,q}(\omega_2)$. We have shown that $u\in W^{s,q}(\omega_2)$. Since by hypothesis, $u\in W_0^{s,2}(\bOm)$, it follows from the Sobolev embedding \eqref{sob-emb1} that $u\in L^{\frac{2N}{N-2s}}(\Omega)\hookrightarrow L^q(\Omega)$. Thus $u\in W^{s,q}(\omega_2)\cap L^q(\Omega)$.
Let
\begin{align*}
g:=u(-\Delta)^s\eta  -I_s(u,\eta).
\end{align*}
Also in this case, it is possible to prove that $g\in L^q(\RR^N)$ and there exists a constant $C>0$ such that
\begin{align}\label{norm-f-q}
\|g\|_{L^q(\RR^N)}\le C\left(\|u\|_{W^{s,q}(\omega_2)}+\|u\|_{L^q(\Omega)}\right).
\end{align}

We omit the proof of \eqref{norm-f-q}; it is totally analogous to the one we made in Step 1. As before, we have proved that $\eta u$ is a weak solution to the Poisson equation \eqref{PE} with $F=\eta(-\Delta)^su+g$. Since $F\in L^q(\RR^N)$ and $q\ge $, it follows from Theorem \ref{re-r-N} that $\eta u\in W^{2s,q}(\RR^N)$. Thus $u\in W_{\rm loc}^{2s,q}(\Omega)$. If $2\le p\le \frac{2N}{N-2s}$, then the proof is finished.

(b) Assume that $p>\frac{2N}{N-2s}$. Since $u\in W_{\rm loc}^{2s,q}(\Omega)$, we have that $u\in W^{2s,q}(\omega_2)$. This implies that $u\in W^{s,q_1}(\omega_2)$ with $q_1:=\min\{p,\frac{Nq}{N-sq}\}=\min\{p,\frac{2N}{N-4s}\}$.  It also follows from Lemma \ref{lem-reg-p} that  $u\in L^{q_1}(\Omega)$. We have shown that $u\in W^{s,q_1}(\omega_2)\cap L^{q_1}(\Omega)$.
Now proceeding as in part (a) we get that $u\in W_{\rm loc}^{2s,q_1}(\Omega)$. Here also if $2\le p\le \frac{2N}{N-4s}$, then the proof is finished. Otherwise, iterating we will get that $u\in W_{\rm loc}^{2s,q_j}(\Omega)$ with $q_j=\min\{p,\frac{2N}{N-sj}\}$ for all $j\ge 2$. Hence, we can find $j\in\NN$, $j\ge 2$,  such that $2\le p\le \frac{2N}{N-sj}$. The proof of Theorem \ref{reg-p} is finished.
\end{proof}

\section{The approach using the heat semigroup representation}\label{fl-heat-sec}

One of the main passages in the proof of Theorems \ref{reg-2} and \ref{reg-p} has been to show that, after having applied the cut-off function $\eta$, the remainder $g$, that we obtain applying \eqref{prod-formu} to the product $\eta u$, belongs to $L^p(\RR^N)$ if $f$ belongs to $L^p(\Omega)$. In this section, we present an alternative proof of this fact, using the characterization of the fractional Laplacian through the heat semigroup introduced in \eqref{fl_heat}. 

The heat equation representation of the operator looks a priori local and this will allow us to give a very precise information on the commutator, in particular in terms of the order of regularity and the localization properties. 

Before going further into our discussion, we first need to describe how the operator introduced in \eqref{fl_heat} behaves when it is applied to the function $\eta u$. For simplicity of notation, let us define 
\begin{align}
	\varrho(t):=e^{t\Delta}(\eta u),\;\;\;t\ge 0.
\end{align}
Then, by definition, we have that $\varrho$ satisfies the following heat equation on $\RR^N$
\begin{align}\label{heat_rho}
	\varrho_t-\Delta \varrho=0, \;\;\;t> 0,\;\;\; \varrho(0)=\eta u.
\end{align}
Furthermore, the solution of \eqref{heat_rho} can be written in the form $\varrho=\phi\eta+z$ with 
\begin{align}\label{heat_phi}
	\phi_t-\Delta \phi=0, \;\;\;t> 0,\;\;\; \phi(0)=u
\end{align}
and
\begin{align}\label{heat_z}
	z_t-\Delta z=2\textrm{div}(\phi\nabla\eta) - \phi\Delta\eta,\;\;\;t> 0,\;\;\; z(0)=0.
\end{align}
Finally, we can trivially compute
\begin{align*}
	(-\Delta)^s(\eta u) &= \frac{1}{\Gamma(-s)} \int_0^{+\infty} \Big(\varrho(t)-\varrho(0)\Big)\frac{dt}{t^{1+s}} 
	\\
	&= \frac{1}{\Gamma(-s)} \int_0^{+\infty} \Big(\eta\phi(t)+z(t)-\eta u(t)\Big)\frac{dt}{t^{1+s}}
	\\
	&=\frac{\eta}{\Gamma(-s)} \int_0^{+\infty} \Big(e^{t\Delta}u-u\Big)\frac{dt}{t^{1+s}} + \frac{1}{\Gamma(-s)} \int_0^{+\infty} \frac{z(t)}{t^{1+s}}\,dt. 
\end{align*}
Therefore we find an expression of the type 
\begin{align}
(-\Delta)^s(\eta u) = \eta(-\Delta)^s u+g,
\end{align}
where the remainder term $g$  is given by
\begin{align}\label{reminder}
	g(x):=\frac{1}{\Gamma(-s)}\int_0^{+\infty}\frac{z(x,t)}{t^{1+s}}\,dt,\;\;x\in\RR^N.
\end{align}

\subsection{Proof of the $L^2$ regularity of $g$}

Keeping in mind the notations that we have just introduced, we can now prove the following result.

\begin{lemma}\label{rem_lemma}
Let $u\in W_0^{s,2}(\bOm)$ and let $\eta$ be the cut-off function introduced in \eqref{eta}. Moreover, let $g$ be the remainder term in the expression 
\begin{align*}
	(-\Delta)^s(\eta u) = \eta(-\Delta)^su+g.
\end{align*} 
Then, there exists a constant $C>0$ (independent of $u$) such that 
\begin{align}\label{norm_est}
	\norm{g}{L^2(\RR^N)}\leq C\norm{u}{W^{s,2}(\Omega)}.
\end{align}
\end{lemma}

\begin{proof}
According to the expression \eqref{reminder}, to estimate the $L^2$-norm of $g$, it will be enough to obtain suitable bounds of the $L^2$-norm of $z$. For this purpose, we notice that the solution of \eqref{heat_z} can be computed explicitly as 
\begin{align}\label{z_sol}
	z(x,t) = \int_0^t\int_{\RR^N} G(x-y,t-\tau)h(y,\tau)\,dyd\tau=\int_0^t \big[G(\cdot,t-\tau)\ast h(\cdot,\tau)\big](x)\,d\tau,\;\;x\in\RR^N,
\end{align}
where $G$ is the Gaussian kernel
\begin{align*}
	G(x,t):= \left(4\pi t\right)^{-\frac{N}{2}}\exp\left(-\frac{|x|^2}{4t}\right),\;\;x\in\RR^N,\;t>0,
\end{align*}
and $h$ is given by $h:=2\textrm{div}(\phi\nabla\eta) - \phi\Delta\eta$. Hence, in particular, we have
\begin{align}\label{z_split}
	z(t) = 2\int_0^t G(t-\tau)\ast \textrm{div}(\phi(\tau)\nabla\eta)\,d\tau - \int_0^t G(t-\tau)\ast (\phi(\tau)\Delta\eta)\,d\tau:=z_1(t)-z_2(t).
\end{align}

In \eqref{z_split}, since we are only interested in the behavior of $z$ with respect to the variable $t$, and for keeping the notations lighter, we have omitted the dependence of $z$ on the variable $x$. We will maintain this convention until the end of the proof. Finally, we have (recall that $\Gamma(1-s)=-s\Gamma(-s)$)
\begin{align}\label{int_z_split}
	\norm{g}{L^2(\RR^N)} \leq & \frac{s}{\Gamma(1-s)}\int_0^{+\infty}\frac{\norm{z(t)}{L^2(\RR^N)}}{t^{1+s}}\,dt \notag\\
	=& \frac{s}{\Gamma(1-s)}\int_0^1\frac{\norm{z(t)}{L^2(\RR^N)}}{t^{1+s}}\,dt + \frac{s}{\Gamma(1-s)} \int_1^{+\infty}\frac{\norm{z(t)}{L^2(\RR^N)}}{t^{1+s}}\,dt \nonumber
	\\
	\leq & \frac{s}{\Gamma(1-s)}\int_0^1\frac{\norm{z_1(t)}{L^2(\RR^N)}}{t^{1+s}}\,dt + \frac{s}{\Gamma(1-s)}\int_0^1\frac{\norm{z_2(t)}{L^2(\RR^N)}}{t^{1+s}}\,dt\notag\\
	& + \frac{s}{\Gamma(1-s)}\int_1^{+\infty}\frac{\norm{z_1(t)}{L^2(\RR^N)}}{t^{1+s}}\,dt \nonumber
	+ \frac{s}{\Gamma(1-s)}\int_1^{+\infty}\frac{\norm{z_2(t)}{L^2(\RR^N)}}{t^{1+s}}\,dt\nonumber		
	\\
	:=& A_1^1 + A_1^2 + A_2^1 + A_2^2.
\end{align}
We proceed now estimating the terms $A_1^1$, $A_1^2$, $A_2^1$ and $A_2^2$ separately. 

\subsubsection*{Step 1. Preliminary estimates.} 
First of all, throughout the remainder of the proof, $C$ will denote a generic positive constant depending only on $\Omega$, $\eta$, $s$ and $N$. This constant may change even from line to line. 

Now, we observe that by using some classical energy estimates for solutions to the heat equation, we obtain that
\begin{align}\label{energy_phi}
	\norm{\phi(t)}{L^2(\RR^N)}\leq\norm{u}{L^2(\Omega)}, 
\end{align}
\begin{align}\label{energy_phi_s}
	\norm{\phi(t)}{W^{s,2}(\RR^N)}\leq C\norm{u}{W^{s,2}(\Omega)},\;\;\;\textrm{ for all }\;s\in(0,1).
\end{align}

These inequalities can be easily proved by multiplying \eqref{heat_phi} by $\phi$ and $(-\Delta)^s\phi$, respectively, and integrating by parts. Moreover, to obtain \eqref{energy_phi_s} we also took into account that, according to \cite[Lemma 16.3]{tartar}, we have 
\begin{align*}
	\norm{(-\Delta)^{\frac{s}{2}}\phi(t)}{L^2(\RR^N)}=C\int_{\RR^N}\int_{\RR^N}\frac{|\phi(x,t)-\phi(y,t)|^2}{|x-y|^{N+2s}}\,dxdy.
\end{align*} 

In our proof, we will also need the following classical property of convolution (see e.g. \cite[Proposition 8.9]{folland}). For all $\varphi_1\in L^{q_1}(\RR^N)$, $\varphi_2\in L^{q_2}(\RR^N)$ and for all $q_1,q_2$ and $q_3$ satisfying
\begin{align}\label{convolution_young_cond}
	1\leq q_1,q_2,q_3<+\infty,\;\;\;\frac{1}{q_1}+\frac{1}{q_2}=\frac{1}{q_3}+1,
\end{align}	
we have that
\begin{align}\label{convolution_young}
	\norm{\varphi_1\ast\varphi_2}{L^{q_3}(\RR^N)}\leq\norm{\varphi_1}{L^{q_1}(\RR^N)}\norm{\varphi_2}{L^{q_2}(\RR^N)}.
\end{align}

This is a straightforward consequence of the Young inequality. Finally, we recall that for all $1\leq p<\infty$ and $k\geq 0$, the function $G$ satisfies the following decay properties (see, e.g. \cite{kato}): there exists a constant $C>0$ such that
\begin{align}\label{gaussian_est}
	\norm{D^kG(t)}{L^p(\RR^N)} \leq Ct^{-\frac{N}{2}\left(1-\frac{1}{p}\right)-\frac{k}{2}}.
\end{align}

Here, $k=(k_1,k_2,\ldots,k_N)$ is a multi-index with modulus $|k|=k_1+k_2+\cdots+k_N$ and we used the classical Schwartz notation 
\begin{align*}
	D^k\phi(x)=\frac{\partial^{|k|}\phi(x)}{\partial x_1^{k_1}\partial x_2^{k_2}\cdots\partial x_N^{k_N}}.
\end{align*}
In particular, we have that
\begin{align}\label{gaussian_est_2}
	\begin{array}{ll}
	\norm{G(t)}{L^2(\RR^N)} \leq Ct^{-\frac{N}{4}}, & \norm{\nabla_x G(t)}{L^2(\RR^N)} \leq C t^{-\frac{N}{4}-\frac{1}{2}}
	\\
	\\
	\norm{(G\ast h)(t)}{L^2(\RR^N)} \leq C \norm{h}{L^2(\Omega)}, & \norm{(\nabla_x G\ast h)(t)}{L^2(\RR^N)} \leq Ct^{-\frac{1}{2}}\norm{h}{L^2(\Omega)}.
	\end{array}
\end{align}

\subsubsection*{Step 2. Upper bound of $A_1:=A_1^1+A_1^2$.} 
We start by estimating the contribution of $z_1$. Using \eqref{convolution_young} with $q_1=1$,  $q_2=q_3=2$, and \eqref{energy_phi_s} we get that
\begin{align*}
	\norm{z_1(t)}{L^2(\RR^N)} &\leq\int_0^t\norm{G(t-\tau)\ast\textrm{div}(\phi(\tau)\nabla\eta)}{L^2(\RR^N)}d\tau \\
	&\leq C\int_0^t\norm{D^{1-s}G(t-\tau)\ast D^s(\phi(\tau)\nabla\eta)}{L^2(\RR^N)}d\tau 
	\\
	&\leq C\int_0^t (t-\tau)^{-\frac{1-s}{2}}\norm{\phi(\tau)\nabla\eta}{W^{s,2}(\RR^N)}d\tau\\
	& \leq C\norm{u}{W^{s,2}(\Omega)}\int_0^t (t-\tau)^{-\frac{1-s}{2}}d\tau =Ct^{\frac{1+s}{2}}\norm{u}{W^{s,2}(\Omega)}.
\end{align*}	

In the previous computations,  $D^s$ denotes the differential operator with Fourier symbol $|\cdot|^s$, that is, $D^s\zeta(\cdot)=\mathcal{F}^{-1}\big\{\,|\cdot|^s\mathcal{F}\zeta(\cdot)\big\}$ for all functions $\zeta$ sufficiently smooth. Concerning the contribution of $z_2$, instead, we have
\begin{align*}
	\norm{z_2(t)}{L^2(\RR^N)} &\leq\int_0^t\norm{G(t-\tau)\ast (\phi(\tau)\Delta\eta)}{L^2(\RR^N)}\,d\tau 
	\\
	&\leq C\int_0^t\norm{\phi(\tau)\Delta\eta}{L^2(\RR^N)}\,d\tau\leq Ct\norm{u}{L^2(\Omega)}.
\end{align*}
Since $0<s<1$, we have that
\begin{align}\label{A_1_est}
	A_1 \leq& C\norm{u}{W^{s,2}(\Omega)}\int_0^1\frac{dt}{t^{\frac{1+s}{2}}} + C\norm{u}{L^2(\Omega)}\int_0^1\frac{dt}{t^s}\notag\\
	\leq &C\norm{u}{W^{s,2}(\Omega)} + C\norm{u}{L^2(\Omega)} \leq C\norm{u}{W^{s,2}(\Omega)}.
\end{align}

\subsubsection*{Step 3. Upper bound of $A_2^1$.} 
We have to distinguish three cases: $N=1$, $N=2$ and $N\geq 3$. 

\paragraph*{\textbf{Case 1: $N=1$}}
Since $u\in L^2(\Omega)$ and $\Omega$ is bounded, we also have $u\in L^1(\Omega)$. Hence, the quantity 
\begin{align*}
	m:=\int_{\RR} u\,dx=\int_{\Omega} u\,dx,
\end{align*}
is well defined.

\indent Let us now rewrite $u=(u-m\delta_0) + m\delta_0$, where $\delta_0$ is the Dirac delta at $x=0$. With this splitting in mind, we have that $\phi$ can be seen as the sum $\phi=\psi+mG$, with $\psi$ solving 
\begin{align}\label{heat_psi}
	\psi_t-\psi_{xx}=0, \;\;\; t>0,\;\;\; \psi(0)=u-m\delta_0.
\end{align}
Therefore, we obtain
\begin{align*}
	z_1(t) = \int_0^t G(t-\tau)\ast(\psi(\tau)\eta_x)_x\,d\tau + \int_0^t G(t-\tau)\ast(mG(\tau)\eta_x)_x\,d\tau:=z_{1,\psi}(t)+z_{1,G}(t).
\end{align*}
Let us consider firstly the term $z_{1,\psi}$. First of all, we notice that $\psi=\theta_x$ with $\theta$ solving
\begin{align}\label{heat_theta}
	\theta_t-\theta_{xx}=0, \;\;\; t>0, \;\;\; \theta(0)=\int_{-\infty}^x (u-m\delta_0)\,d\xi,
\end{align}
and therefore, 
\begin{align*}
	z_{1,\psi}(t) = \int_0^t G(t-\tau)\ast(\theta_x(\tau)\eta_x)_x\,d\tau.
\end{align*}
Now
\begin{align*}
	\norm{z_{1,\psi}(t)}{L^2(\RR)} & \leq \int_0^t \norm{G(t-\tau)\ast(\theta_x(\tau)\eta_x)_x}{L^2(\RR)}\,d\tau 
	\\
	&= \int_0^t \norm{G_x(t-\tau)\ast(\theta_x(\tau)\eta_x)}{L^2(\RR)}\,d\tau 
	\\
	&\leq \int_0^t (t-\tau)^{-\frac{3}{4}}\norm{\theta_x(\tau)\eta_x}{L^1(\RR)}\,d\tau.
\end{align*}
Moreover, we have
\begin{align*}
	\norm{\theta_x(\tau)\eta_x}{L^1(\RR)} \leq C\norm{\theta_x(\tau)}{L^1(\Omega)}\leq C\tau^{-\frac{1}{2}}\norm{\theta(0)}{L^1(\Omega)}\leq C\tau^{-\frac{1}{2}}\norm{u}{L^2(\Omega)},
\end{align*}
where the last inequality is justified by the fact that the initial datum $\theta(0)$ is well defined as an $L^1$-function compactly supported in $\Omega$, and there exists a constant $C>0$ such that 
\begin{align*}
	\norm{\theta(0)}{L^1(\Omega)}\leq C\norm{u}{L^2(\Omega)}.
\end{align*} 
See \cite[Theorem 1]{DZ} for more details. Hence,
\begin{align*}
	\norm{z_{1,\psi}(t)}{L^2(\RR)} &\leq C\norm{u}{L^2(\Omega)}\int_0^t (t-\tau)^{-\frac{3}{4}}\tau^{-\frac{1}{2}}\,d\tau = Ct^{-\frac{1}{4}}\norm{u}{L^2(\Omega)}.
\end{align*}
Let us now analyze the term $z_{1,G}$ which, we remind, is defined as 
\begin{align*}
	z_{1,G}(t) = m\int_0^t G(t-\tau)\ast(G(\tau)\eta_x)_x\,d\tau.
\end{align*}
We have
\begin{align*}
	\norm{z_{1,G}(t)}{L^2(\RR)} &\leq m\int_0^t \norm{G(t-\tau)\ast(G(\tau)\eta_x)_x}{L^2(\RR)}\,d\tau = m\int_0^t \norm{G_x(t-\tau)\ast(G(\tau)\eta_x)}{L^2(\RR)}\,d\tau.
\end{align*}
Now, since $u$ is compactly supported in $\Omega$, the Cauchy-Schwarz inequality yields 
\begin{align*}
	m\leq\norm{u}{L^1(\Omega)}\leq\sqrt{|\Omega|}\norm{u}{L^2(\Omega)},
\end{align*}	
where $|\Omega|$ is the Lebesgue measure of $\Omega$; hence 
\begin{align*}
	\norm{z_{1,G}(t)}{L^2(\RR)} &\leq C\norm{u}{L^2(\Omega)}\int_0^t \norm{G_x(t-\tau)\ast(G(\tau)\eta_x)}{L^2(\RR)}\,d\tau. 
\end{align*}
Rewrite $G(\tau)\eta_x=(G(\tau)\eta)_x-G_x(\tau)\eta$. Then
\begin{align*}
	\norm{z_{1,G}(t)}{L^2(\RR)} \leq& C\norm{u}{L^2(\Omega)}\int_0^t \norm{G_x(t-\tau)\ast(G(\tau)\eta)_x}{L^2(\RR)}\,d\tau 
	\\
	&+ C\norm{u}{L^2(\Omega)}\int_0^t \norm{G_x(t-\tau)\ast(G_x(\tau)\eta)}{L^2(\RR)}\,d\tau := J_1 + J_2.
\end{align*}
Concerning $J_1$ we have 
\begin{align*}
	J_1 &\leq C\norm{u}{L^2(\Omega)}\int_0^t \norm{D^{1-s}G_x(t-\tau)\ast D^s(G(\tau)\eta)}{L^2(\RR)}\,d\tau
	\\
	&\leq C\norm{u}{L^2(\Omega)}\int_0^t \norm{D^{1-s}G_x(t-\tau)}{L^1(\RR)}\norm{D^s(G(\tau)\eta)}{L^2(\RR)}\,d\tau 
	\\
	&\leq C\norm{u}{L^2(\Omega)}\int_0^t (t-\tau)^{-\frac{2-s}{2}}\tau^{-\frac{1}{4}-\frac{s}{2}}\,d\tau = Ct^{-\frac{1}{4}}\norm{u}{L^2(\Omega)}.
\end{align*}
Finally, for $J_2$ we have 
\begin{align*}
	J_2 &\leq C\norm{u}{L^2(\Omega)}\int_0^t \norm{G_x(t-\tau)}{L^2(\RR)}\norm{G_x(\tau)\eta_x}{L^1(\RR)}\,d\tau  
	\\
	&\leq C\norm{u}{L^2(\Omega)}\int_0^t (t-\tau)^{-\frac{3}{4}}\tau^{-\frac{1}{2}}\,d\tau = C t^{-\frac{1}{4}}\norm{u}{L^2(\Omega)}.
\end{align*}
Summarizing we get that
\begin{align*}
	\norm{z_{1,G}(t)}{L^2(\RR)} &\leq Ct^{-\frac{1}{4}}\norm{u}{L^2(\Omega)}
\end{align*}
which, combined with the estimate that we have obtained before for $z_{1,\psi}$ gives
\begin{align*}
	\norm{z_1(t)}{L^2(\RR)} &\leq Ct^{-\frac{1}{4}}\norm{u}{L^2(\Omega)}.
\end{align*}
Therefore, since $s>0$, we finally get that
\begin{align*}
	A_2^1=\frac{s}{\Gamma(1-s)}\int_1^{+\infty}\frac{\norm{z_1(t)}{L^2(\RR)}}{t^{1+s}}\,dt\leq C\norm{u}{L^2(\Omega)}\int_1^{+\infty}\frac{dt}{t^{s+\frac{5}{4}}}=C\norm{u}{L^2(\Omega)}.
\end{align*}

\paragraph*{\textbf{Case 2: $N=2$}}
Using again \eqref{energy_phi_s}, \eqref{convolution_young} with $q_1=q_3=2$ and $q_2=1$ and the fact that $\eta$ has compact support, we get that
\begin{align}\label{z12_est_2}
	\norm{z_1(t)}{L^2(\RR^2)} &\leq 2\int_0^t \norm{G(t-\tau)\ast\textrm{div}(\phi(\tau)\nabla\eta)}{L^2(\RR^2)}\,d\tau \notag\\
	&\le 2\int_0^t \norm{D^{1-s}G(t-\tau)\ast D^s(\phi(\tau)\nabla\eta)}{L^2(\RR^2)}\,d\tau \nonumber
	\\
	&\leq C\norm{u}{W^{s,2}(\Omega)}\int_0^t (t-\tau)^{-1+\frac{s}{2}}\,d\tau \leq Ct^{\frac{s}{2}}\norm{u}{W^{s,2}(\Omega)}.
\end{align}
Since $s>0$, it follows that
\begin{align*}
	A_2^1=\frac{s}{\Gamma(1-s)}\int_1^{+\infty}\frac{\norm{z_1(t)}{L^2(\RR^2)}}{t^{1+s}}\,dt\leq C\norm{u}{W^{s,2}(\Omega)}\int_1^{+\infty}\frac{dt}{t^{1+\frac{s}{2}}}=C\norm{u}{W^{s,2}(\Omega)}.
\end{align*}

\paragraph*{\textbf{Case 3: $N\geq 3$}}
This case is more delicate and we need to proceed in a slightly different way. For a given $\varepsilon\in [0,1]$, we will apply again \eqref{convolution_young} but this time by choosing 
\begin{align}\label{pqr_choice}
	q_1=\frac{4-2\varepsilon}{4-3\varepsilon},\;\; q_2=2-\varepsilon,\;\; q_3=2.
\end{align}

It is straightforward to check that $q_1$, $q_2$ and $q_3$ given in \eqref{pqr_choice} satisfy condition \eqref{convolution_young_cond}. In particular, we notice that $q_2\in[1,2]$. With this particular choice of the parameters we have
\begin{align}\label{z12_est_N}
	\norm{z_1(t)}{L^2(\RR^2)} &\leq 2\int_0^t \norm{G(t-\tau)\ast\textrm{div}(\phi(\tau)\nabla\eta)}{L^2(\RR^N)}\,d\tau \notag\\
	&= 2\int_0^t \norm{D^{1-s}G(t-\tau)\ast D^s(\phi(\tau)\nabla\eta)}{L^2(\RR^N)}\,d\tau \nonumber
	\\
	&\leq C\int_0^t (t-\tau)^{-\frac{N}{2}\frac{\varepsilon}{4-2\varepsilon}-\frac{1-s}{2}}\norm{\phi(\tau)\nabla\eta}{L^{2-\varepsilon}(\RR^N)}\,d\tau \notag\\
	&\leq C t^{\frac{1+s}{2}-\frac{N}{2}\frac{\varepsilon}{4-2\varepsilon}}\norm{u}{W^{s,2}(\Omega)},
\end{align}
provided that
\begin{align*}
	\frac{1+s}{2}-\frac{N}{2}\frac{\varepsilon}{4-2\varepsilon}>0\;\;\;\Longrightarrow \varepsilon<\frac{4+4s}{N+2+2s}.
\end{align*}
Therefore, 
\begin{align*}
	A_2^1=\frac{s}{\Gamma(1-s)}\int_1^{+\infty}\frac{\norm{z_1(t)}{L^2(\RR^N)}}{t^{1+s}}\,dt\leq C\norm{u}{W^{s,2}(\Omega)}\int_1^{+\infty}\frac{dt}{t^{\frac{1+s}{2}+\frac{N}{2}\frac{\varepsilon}{4-2\varepsilon}}}=C\norm{u}{W^{s,2}(\Omega)},
\end{align*}
if we impose that 
\begin{align*}
	\frac{1+s}{2}+\frac{N}{2}\frac{\varepsilon}{4-2\varepsilon}>1\;\;\;\Longrightarrow \varepsilon>\frac{4-4s}{N+2-2s}.
\end{align*}
Thus, we obtain a further condition on $\varepsilon$, namely
\begin{align*}
		\varepsilon\in\left(\frac{4-4s}{N+2-2s},\frac{4+4s}{N+2+2s}\right).
\end{align*}
Furthermore, we can easily check that, for all $s\in(0,1)$ and $N\geq 3$ the set 
\begin{align}\label{espilon_set}
	[0,1]\cap\left(\frac{4-4s}{N+2-2s},\frac{4+4s}{N+2+2s}\right)\ne\emptyset.
\end{align}
Therefore, for any given $s\in(0,1)$ and $N\geq 3$, we can always choose $q_1$, $q_2$ and $q_3$ as in \eqref{pqr_choice} such that 
\begin{align*}
	A_2^1\leq C\norm{u}{W^{s,2}(\Omega)}.
\end{align*}

\subsubsection*{Step 4. Upper bound of $A_2^2$.} 
Using again \eqref{convolution_young}, this time with $q_1=1$, $q_2=q_3=2$ and the fact that $\eta$ has compact support, for a given $\alpha\in(2-2s,2)$ we can estimate
\begin{align}\label{z2_est}
	\norm{z_2(t)}{L^2(\RR^N)} &\leq\int_0^t\norm{G(t-\tau)\ast(\phi(\tau)\Delta\eta)}{L^2(\RR^N)}\,d\tau\notag\\
	 &\le \int_0^t\norm{D^{\alpha}G(t-\tau)\ast D^{-\alpha}(\phi(\tau)\Delta\eta)}{L^2(\RR^N)}\,d\tau	\nonumber
	\\
	&\leq \int_0^t\norm{D^{\alpha}G(t-\tau)}{L^1(\RR^N)}\norm{D^{-\alpha}(\phi(\tau)\Delta\eta)}{L^2(\RR^N)}\,d\tau \notag\\
	&\leq C\int_0^t (t-\tau)^{-\frac{\alpha}{2}}\norm{\phi(\tau)\Delta\eta}{W^{-\alpha,2}(\RR^N)}\,d\tau \nonumber
	\\
	&\leq Ct^{1-\frac{\alpha}{2}}\norm{u}{L^2(\Omega)}.
\end{align}
Hence
\begin{align*}
	A_2^2 \leq C\norm{u}{L^2(\Omega)}\int_1^{+\infty}\frac{dt}{t^{s+\frac{\alpha}{2}}} = C\norm{u}{L^2(\Omega)}.
\end{align*}

\subsubsection*{Step 5. Conclusion} 
Collecting all the above estimates, we can finally conclude that there exists a constant $C>0$ such that \eqref{norm_est} holds, and the proof of Lemma \ref{rem_lemma} is finished.
\end{proof}

\subsection{Proof of the $L^p$ regularity of $g$}

Lemma \ref{rem_lemma} provides an alternative proof of the $L^2(\RR^N)$ regularity of the remainder term $g$ which appears in the formula for the fractional Laplacian of the product $\eta u$. Moreover, as we did before in Section \ref{local-reg-sec}, also this result can be generalized to the $L^p$ setting. In particular, we can prove the following.

\begin{lemma}\label{rem_lemma_p}
Let $u\in W_0^{s,2}(\bOm)$, $p\geq 2$, $N\geq 2$ and let $\eta$ be the cut-off function introduced in \eqref{eta}.  Moreover, let $g$ be the remainder term in the expression 
\begin{align*}
	(-\Delta)^s(\eta u) = \eta(-\Delta)^su+g.
\end{align*} 
Then, there exists a constant $C>0$ (independent of $u$) such that 
\begin{align}\label{norm_est_r}
	\norm{g}{L^p(\RR^N)}\leq C\left(\norm{u}{L^p(\Omega)} + \norm{u}{W^{s,2}(\Omega)}\right).
\end{align}
\end{lemma}

\begin{proof}
We recall that, according to \eqref{reminder}, to estimate the $L^p$-norm of $g$ we only need an appropriate bound for the $L^p$-norm of the function $z$ introduced in \eqref{z_sol}. Moreover, also in this case we have $\norm{g}{L^p(\RR^N)}\leq A_1^1+A_1^2+A_2^1+A_2^2$, where, with some abuse of notations, the terms $A_1^1$, $A_1^2$, $A_2^1$ and $A_2^2$ are the same ones as in \eqref{int_z_split}, after having replaced $\norm{z(t)}{L^2(\RR^N)}$ with $\norm{z(t)}{L^p(\RR^N)}$. 

\subsubsection*{Step 1. Preliminary estimates.} 

First of all, we recall that from Proposition \ref{prop-ex}, it follows that $u\in W_0^{s,2}(\bOm)$ and it follows from Lemma \ref{lem-reg-p} that $u\in L^p(\Omega)$.

Moreover, we observe that the classical energy decay estimates presented in \eqref{energy_phi} can be generalized to the $L^p$ setting. In particular we have 
\begin{align}\label{energy_phi_p}
	\norm{\phi(t)}{L^p(\RR^N)}\leq\norm{u}{L^p(\Omega)}.
\end{align}

The proof of \eqref{energy_phi_p} is a straightforward application of \eqref{convolution_young}, taking into account the fact that the solution of the heat equation \eqref{heat_phi} is given by the convolution $\phi(t)=G(t)\ast u$.

\subsubsection*{Step 2. Upper bound of $A_1^1$.} 
First of all, throughout the remainder of the proof, $C$ will denote a generic positive constant depending only on $\Omega$, $\eta$, $s$, $p$ and $N$. This constant may change even from line to line. 


\noindent Now, using \eqref{convolution_young} with $q_1=2p/(2+p)$, $q_2=2$ and $q_3=p$ we get that
\begin{align*}
	\norm{z_1(t)}{L^p(\RR^N)} &\leq\int_0^t\norm{G(t-\tau)\ast\textrm{div}(\phi(\tau)\nabla\eta)}{L^p(\RR^N)}d\tau \\
	&\leq C\int_0^t\norm{D^{1-s}G(t-\tau)\ast D^s(\phi(\tau)\nabla\eta)}{L^p(\RR^N)}d\tau 
	\\
	&\leq C\int_0^t\norm{D^{1-s}G(t-\tau)}{L^{q_1}(\RR^N)}\norm{D^s(\phi(\tau)\nabla\eta)}{L^2(\RR^N)}d\tau 
	\\	
	& \leq C\norm{u}{W^{s,2}(\Omega)}\int_0^t (t-\tau)^{-\frac{N}{2}\left(1-\frac{1}{q_1}\right)-\frac{1-s}{2}}d\tau =Ct^{\frac{1+s}{2}-\frac{N}{2}\left(1-\frac{1}{q_1}\right)}\norm{u}{W^{s,2}(\Omega)},
\end{align*}	
provided that 
\begin{align*}
	\frac{1+s}{2}-\frac{N}{2}\left(1-\frac{1}{q_1}\right)>0 \;\;\;\Longrightarrow\;\;\; q_1<\frac{N}{N-1-s}.
\end{align*}


\noindent In view of the previous estimate, we have

\begin{align*}
	A_1^1\leq C \norm{u}{W^{s,2}(\Omega)} \int_0^1 \frac{dt}{t^{\frac{1+s}{2}+\frac{N}{2}\left(1-\frac{1}{q_1}\right)}} = C \norm{u}{W^{s,2}(\Omega)},
\end{align*}
provided that
\begin{align*}
	\frac{1+s}{2}+\frac{N}{2}\left(1-\frac{1}{q_1}\right)<1 \;\;\;\Rightarrow\;\;\; q_1<\frac{N}{N-1+s}.
\end{align*}

Finally, we notice that, by hypothesis we have $p\geq 2$; this, according to the definition of $q_1$ that we are considering, corresponds to the further condition $1\leq q_1<2$. Hence, recollecting the conditions on $q_1$ that we have encountered we conclude that we have to impose 
\begin{align*}
	1\leq q_1 < \min\left\{2,\frac{N}{N-1+s},\frac{N}{N-1-s}\right\}=\frac{N}{N-1+s}=1+\frac{1-s}{N-1+s}
\end{align*}
Summarizing, we have 
\begin{align*}
	A_1^1\leq C \norm{u}{W^{s,2}(\Omega)},
\end{align*}
if in our computations we assume 
\begin{align*}
	1\leq q_1 < 1+\frac{1-s}{N-1+s}.
\end{align*}

\subsubsection*{Step 3. Upper bound of $A_1^2$.}
We have
\begin{align*}
	\norm{z_2(t)}{L^p(\RR^N)} &\leq\int_0^t\norm{G(t-\tau)\ast (\phi(\tau)\Delta\eta)}{L^p(\RR^N)}\,d\tau \\
	&\leq C\int_0^t\norm{\phi(\tau)\Delta\eta}{L^p(\RR^N)}\,d\tau\leq Ct\norm{u}{L^p(\Omega)}.
\end{align*}
Since $0<s<1$, we have that
\begin{align}\label{A_1_est_r}
	A_1^2 \leq C\norm{u}{L^p(\Omega)}\int_0^1\frac{dt}{t^s} \leq C\norm{u}{L^p(\Omega)}.
\end{align}

\subsubsection*{Step 4. Upper bound of $A_2^1$.} 
Repeating the same computations that we did in Step 2, we get that
\begin{align*}
	\norm{z_1(t)}{L^p(\RR^N)} Ct^{\frac{1+s}{2}-\frac{N}{2}\left(1-\frac{1}{q_1}\right)}\norm{u}{W^{s,2}(\Omega)},
\end{align*}	
provided that 
\begin{align*}
	\frac{1+s}{2}-\frac{N}{2}\left(1-\frac{1}{q_1}\right)>0 \;\;\;\Longrightarrow\;\;\; q_1<\frac{N}{N-1-s}.
\end{align*}
Therefore
\begin{align*}
	A_2^1\leq C \norm{u}{W^{s,2}(\Omega)} \int_1^{+\infty} \frac{dt}{t^{\frac{1+s}{2}+\frac{N}{2}\left(1-\frac{1}{q_1}\right)}} = C \norm{u}{W^{s,2}(\Omega)},
\end{align*}
provided that
\begin{align*}
	\frac{1+s}{2}+\frac{N}{2}\left(1-\frac{1}{q_1}\right)>1 \;\;\;\Longrightarrow\;\;\; q_1>\frac{N}{N-1+s}.
\end{align*}

Finally, we notice that, by hypothesis we have $p\geq 2$; this, according to the definition of $q_1$ that we are considering, corresponds to the further condition $1\leq q_1<2$. Hence, recollecting the conditions on $q_1$ that we encountered we conclude that we have to impose 
\begin{align*}
	q_1\in\Bigg[1, \min\left\{2,\frac{N}{N-1-s}\right\}\Bigg)\cap\Bigg(\frac{N}{N-1+s}, +\infty\Bigg)=\Bigg(\frac{N}{N-1+s},\min\left\{2,\frac{N}{N-1-s}\right\}\Bigg).
\end{align*}
Summarizing, we have 
\begin{align*}
	A_1^2\leq C \norm{u}{W^{s,2}(\Omega)},
\end{align*}
if in our computations we assume 
\begin{align*}
	q_1\in \Bigg(\frac{N}{N-1+s},\min\left\{2,\frac{N}{N-1-s}\right\}\Bigg).
\end{align*}

\subsubsection*{Step 5. Upper bound of $A_2^2$.} 
Using again \eqref{convolution_young}, this time with $q_1=1$, $q_2=q_3=p$ and the fact that $\eta$ has compact support, for a given $\alpha\in(2-2s,2)$ we can estimate
\begin{align}\label{z2_est_r}
	\norm{z_2(t)}{L^p(\RR^N)} &\leq\int_0^t\norm{G(t-\tau)\ast(\phi(\tau)\Delta\eta)}{L^p(\RR^N)}\,d\tau\notag\\
	 &\le \int_0^t\norm{D^{\alpha}G(t-\tau)\ast D^{-\alpha}(\phi(\tau)\Delta\eta)}{L^p(\RR^N)}\,d\tau	\nonumber
	\\
	&\leq \int_0^t\norm{D^{\alpha}G(t-\tau)}{L^1(\RR^N)}\norm{D^{-\alpha}(\phi(\tau)\Delta\eta)}{L^p(\RR^N)}\,d\tau \notag\\
	&\leq C\int_0^t (t-\tau)^{-\frac{\alpha}{2}}\norm{\phi(\tau)\Delta\eta}{W^{-\alpha,p}(\RR^N)}\,d\tau \nonumber
	\\
	&\leq Ct^{1-\frac{\alpha}{2}}\norm{u}{L^p(\Omega)}.
\end{align}
Hence
\begin{align*}
	A_2^2 \leq C\norm{u}{L^p(\Omega)}\int_1^{+\infty}\frac{dt}{t^{s+\frac{\alpha}{2}}} = C\norm{u}{L^p(\Omega)}.
\end{align*}

\subsubsection*{Step 6. Conclusion} 
Recollecting all the above estimates, we can finally conclude that there exists a constant $C>0$ such that \eqref{norm_est_r} holds. The proof of Lemma \ref{rem_lemma_p} is finished.
\end{proof}

\begin{remark}
Lemma \ref{rem_lemma_p} provides an alternative proof of the $L^p(\RR^N)$-regularity of the remainder term $g$ which appears in the formula for the fractional Laplacian of the product $\eta u$. However, in its proof, we are able to deal only with the case $p>2$ and $N\geq 2$ and $p>2$. When $N=1$ or $1<p<2$, instead, we encounter some difficulties that, at the present stage, we are not able to overcome. We will present these difficulties with more details in Section \ref{open-pb}, dedicated to open problems and perspectives. Nevertheless, we do not exclude that this regularity Lemma could be extended also to the case of one-space dimension.
\end{remark}

\section{Open problems and perspectives}\label{open-pb}

In the present paper we proved that weak solutions to the Dirichlet problem for the fractional Laplacian with a non-homogeneous right hand side $f\in L^p(\Omega)$ ($1<p<\infty$) belong to $W^{2s,p}_{\textrm{loc}}(\Omega)$. 

\noindent The following comments are worth considering.

\begin{enumerate}
\item In the proof of Lemma \ref{rem_lemma_p}, which provides the $L^p(\RR^N)$-regularity of the remainder term $g$ following the approach that employs the heat kernel characterization of the fractional Laplacian, we were not able to treat the cases $1<p<2$ and $N=1$. In more detail, we cannot encounter appropriate bounds for the terms $A_1^1$ and $A_1^2$ (see \eqref{z_split} for more details on the notation). These difficulties are most likely related to the fact that, in this lower dimension case or for lower values of $p$, there is less diffusion and the decay rates that we shall employ are slower. On the other hand, we believe that there has to be a way to solve this problem.

\item A natural interesting extension would  be the analysis  of the global elliptic regularity for weak solutions to  \eqref{DP}. The problem  is delicate however.

For the classical Dirichlet problem associated with the Laplace operator (the case $s=1$), it is well-known that if $\Omega$ is smooth, say of class $C^2$, then weak solutions to the associated problem belong to $W^{2,p}(\Omega)$. 

But, unfortunately, this maximal global elliptic regularity is not true for the fractional Laplacian. To be more precise, assume that $f\in L^p(\Omega)$ ($1<p<\infty$) and let $u$ be the associated weak solution to the Dirichlet problem \eqref{DP}. It is known that, if $p\ge 2$, then $u$ does not always belongs to $W^{2s,p}(\Omega)$ and, if $1<p<2$, then $u$ does not always belong to $B_{p,2}^{2s}(\Omega)$.

If this were the case, then for large $p$ and $\frac 12<s<1$, weak solutions would be at least $\beta$-H\"older  continuous up to the boundary of $\Omega$ of order $\beta>s$. One can see that the latter property is not true by applying the Pohozaev identity obtained in \cite{RS-PI} to the eigenfunctions of the Dirichlet fractional Laplacian. Indeed, let $\lambda_k>0$ be an eigenvalue of $A_D$ and $u_k$ the associated eigenfunction. Then, rewriting the identity in \cite[Proposition 1.6]{RS-PI} with $u_k$ by using the fact that $A_Du_k=\lambda_ku_k$, we get 
\begin{align}\label{poh}
\lambda_k\int_{\Omega}u_n(x\cdot\nabla u_k)\;dx=\frac{2s-N}{2}\lambda_k\int_{\Omega}u_k^2\;dx-\frac{\Gamma(s+1)^2}{2}\int_{\pOm}\left(\frac{u_k}{\rho^s}\right)^2\left(x\cdot\nu\right)\;d\sigma.
\end{align}

Integrating the term in the left-hand side of \eqref{poh} by parts and using that $u_k=0$ on $\pOm$, we get that
\begin{align}\label{poh-2}
s\lambda_k\int_{\Omega}u_k^2\;dx=\frac{\Gamma(s+1)^2}{2}\int_{\pOm}\left(\frac{u_k}{\rho^s}\right)^2\left(x\cdot\nu\right)\;d\sigma.
\end{align}

Now if $u_k$ were $\beta$-H\"older continuous up to the boundary $\pOm$ of order $\beta>s$, then since $0<s<1$ and $s\lambda_k>0$, it would follow from \eqref{poh-2} that $\int_{\Omega}u_k^2\;dx=0$. Thus $u_k=0$ on $\Omega$, which contradicts the fact that $u_k$ is an eigenfunction. We have shown that $u_k$ cannot be $\beta$-H\"older continuous up to the boundary $\pOm$ of order $\beta>s$.

A direct proof that $u_k$ cannot be Lipschitz continuous up to the boundary is also contained in \cite{Val} and the references therein, where it has been shown that the eigenfunctions are $s$-H\"older continuous up to the boundary and this regularity is optimal. Finally, a concrete example, valid for all $1<p<\infty$, has been given in \cite[Section 7]{RS-ES}.

\item It has been shown in \cite{ROS} that if $f\in L^\infty(\Omega)$ with $\Omega$ of class $C^2$ and $u$ is a weak solution of \eqref{DP}, then $u\in C^{0,s}(\RR^N)$ and the function $\rho^{-s}u$, where $\rho=\textrm{dist}(x,\partial\Omega)$ is the distance of a point $x$ to the boundary of the domain $\Omega$, belongs to $C^{0,\alpha}(\bOm)$ for some $0<\alpha<\min\{s,1-s\}$. In addition one has the following precise regularity.
\begin{itemize}
\item If $\Omega$ is of class $C^\infty$ and $f\in C^\infty(\bOm)$, then $\rho^{-s}u\in C^\infty(\bOm)$ (see e.g. \cite{ROS}).

\item If $\Omega$ is of class $C^{2,\beta}$ and $f\in C^\beta(\bOm)$, then $\rho^{-s}u\in C^{s+\beta}(\bOm)$ (see e.g. \cite{RS-FN}).
\end{itemize}

Roughly speaking, these results just mentioned tell us that, if the domain $\Omega$ is regular enough, the solution $u$ to \eqref{DP} can be seen as $u=\rho^s v$, where $v$ is a function regular up to the boundary. 

By part (b), weak solutions are in general not in $W^{2s,p}(\Omega)$, if $p\geq 2$, or in $B_{p,2}^{2s}(\Omega)$, if $1<p<2$. Nevertheless, compared with the above mentioned results, one could expect both $\rho^{-s}u$ and $\rho^{1-s}u$ to be smooth in the $L^p(\Omega)$ context, i.e. to belong to $B^{2s}_{p,2}$, if $1<p<2$, or to $W^{2s,p}(\Omega)$, if $p\ge 2$. 
In view of this, it would be natural to analyze whether this regularity property, which is not available in the literature, is actually true.
Finally, more generally, it is also interesting to investigate for which $\beta>0$ we have the same kind of regularity for the function $\rho^\beta u$. Following our approach we think that it is possible to show that, for every $\beta>s$, $\rho^\beta u$ belongs either to $B^{2s}_{p,2}$, if $1<p<2$, or to $W^{2s,p}(\Omega)$, if $p\ge 2$. However, the most interesting case is $0<\beta\le s$.  We mention that in this situation we have that $\rho^{\beta}u$ is also a solution of a certain Dirichlet problem.
\end{enumerate}

{\appendix\section{}
\label{appendix}
For the sake of completeness,  we introduce some well-known facts about the fractional order Sobolev spaces, which are not so familiar as the classical integral order Sobolev spaces. Let $\Omega\subset\RR^N$ be an arbitrary open set. For $p\in \lbrack
1,\infty )$ and $s\in (0,1)$, we denote by
\begin{equation*}
W^{s,p}(\Omega ):=\left\{ u\in L^{p}(\Omega):\;\int_{\Omega}\int_{\Omega }
\frac{|u(x)-u(y)|^{p}}{|x-y|^{N+ps}}dxdy<\infty \right\},
\end{equation*}
the fractional order Sobolev space endowed with the norm
\begin{equation*}
\Vert u\Vert _{W^{s,p}(\Omega )}:=\left( \int_{\Omega }|u|^{p}\;dx+\int_{\Omega }\int_{\Omega }\frac{|u(x)-u(y)|^{p}}{|x-y|^{N+ps}}
dxdy\right) ^{\frac{1}{p}}.
\end{equation*}
We set
\begin{align*}
W_0^{s,p}(\Omega):=\overline{\mathcal D(\Omega)}^{\;W^{s,p}(\Omega)},
\end{align*}
where $\mathcal D(\Omega)$ is the space of all continuously infinitely differentiable functions with compact support in $\Omega$.

\noindent The following result is taken from \cite[Theorem 1.4.2.4 p.25]{Gris}.
\begin{theorem}\label{theo-Gris}
Let $\Omega\subset\RR^N$ be a bounded open set with Lipschitz continuous boundary and $1<p<\infty$. Then for every $0<s\le \frac 1p$, we have that $W^{s,p}(\Omega)=W_0^{s,p}(\Omega)$ with equivalent norm.
\end{theorem}

It is well-known (see e.g. \cite{NPV,Gris}) that if $\Omega\subset\RR^N$ is a bounded open set with a Lipschitz continuous boundary then
\begin{align}\label{sob-emb1}
W^{s,p}(\Omega)\hookrightarrow L^{q}(\Omega) \;\mbox{ with }\;\;\left\{\begin{array}{ll}
	1\le q\le\frac{Np}{N-sp} &\;\mbox{ if }\;N>sp,
	\\
	1\le q<\infty & \;\mbox{ if }\; N=sp.
\end{array} \right.
\end{align}
If $N<sp$, then
\begin{align}\label{sob-emb2}
W^{s,p}(\Omega)\hookrightarrow C^{0,s-\frac Np}(\bOm).
\end{align}
Next, for $1<p<\infty$ and $0<s<1$ we define
\begin{align*}
W_0^{s,p}(\bOm):=\Big\{u\in W^{s,p}(\RR^N):\; u=0\;\mbox{ on }\;\RR^N\setminus\Omega\Big\}.
\end{align*}

It has been shown in \cite[Lemma 6.1]{NPV} that for an arbitrary bounded open set $\Omega\subset\RR^N$, there exists a constant $C>0$ such that
\begin{align}\label{bound-dist}
\int_{\RR^N\setminus\Omega}\frac{dy}{|x-y|^{N+sp}}\ge C|\Omega|^{-\frac{sp}{N}}.
\end{align}
Using \eqref{bound-dist} we get that there exists a constant $C>0$ such that for every $u\in W_0^{s,p}(\bOm)$,
\begin{align}\label{Nor-1}
\int_{\RR^N}|u|^p\;dx=&\int_{\Omega}|u|^p\;dx\le C\int_{\RR^N}|u(x)|^p\int_{\RR^N\setminus\Omega}\frac{dy}{|x-y|^{N+sp}}\notag\\\le& C \int_{\RR^N }\int_{\RR^N }\frac{|u(x)-u(y)|^{p}}{|x-y|^{N+ps}}dxdy.
\end{align}
It follows from \eqref{Nor-1} that for every $1<p<\infty$ and $0<s<1$, 
\begin{align}\label{equi-nrom}
\|u\|_{W_0^{s,p}(\bOm)}=\left(\int_{\RR^N }\int_{\RR^N }\frac{|u(x)-u(y)|^{p}}{|x-y|^{N+ps}}dxdy\right) ^{\frac{1}{p}},
\end{align}
defines an equivalent norm on $W_0^{s,p}(\bOm)$. We shall denote by $W^{-s,p'}(\bOm)$ the dual of the reflexive Banach space $W_0^{s,p}(\bOm)$, that is,
\begin{align*}
W^{-s,p'}(\bOm):=(W_0^{s,p}(\bOm))^\star\;\;\mbox{ where }\; p':=\frac{p}{p-1}.
\end{align*}

We remark that there is no obvious inclusion between $W_0^{s,p}(\Omega)$ and $W_0^{s,p}(\bOm)$. In fact, for an arbitrary bounded open set $\Omega\subset\RR^N$, the two spaces are different, since $\mathcal D(\Omega)$ is not always dense in $W_0^{s,p}(\bOm)$ (see e.g. \cite{Val-D}). But if $\Omega$ has a continuous boundary, then by \cite[Theorem 6]{Val-D}, $\mathcal D(\Omega)$ is dense in $W_0^{s,p}(\bOm)$ and in addition we have that
\begin{align}\label{eq-sp}
W_0^{s,p}(\bOm)=W_0^{s,p}(\Omega)\;\mbox{ for every }\;\frac 1p<s<1.
\end{align}

In fact, \eqref{eq-sp} follows by using the Hardy inequality for fractional order Sobolev spaces and the following estimate (see e.g. \cite[Formula (1.3.2.12)]{Gris}): there exist two constants $0<C_1\le C_2$ such that
\begin{align}\label{pro-dist}
\frac{C_1}{(\rho(x))^{ps}}\le\int_{\RR^N\setminus\Omega}\frac{dy}{|x-y|^{N+sp}}\le \frac{C_2}{(\rho(x))^{ps}},\;\;\;x\in\Omega.
\end{align}
where $\rho(x):=\mbox{dist}(x,\pOm),\;\;x\in\Omega.$

We also notice that the continuous embeddings \eqref{sob-emb1} and \eqref{sob-emb2} hold with $W^{s,p}(\Omega)$ replaced with $W_0^{s,p}(\Omega)$ or $W_0^{s,p}(\bOm)$ and this case without any regularity assumption on the open set $\Omega$.

Next, if $s>1$ and is not an integer, then we write $s=m+\sigma$ where $m$ is an integer and $0<\sigma<1$. In this case
\begin{align*}
W^{s,p}(\Omega):=\Big\{u\in W^{m,p}(\Omega):\; D^\alpha u\in W^{\sigma,p}(\Omega)\;\mbox{ for any }\;\alpha\;\mbox{ such that }\;|\alpha|=m\Big\}.
\end{align*}
Then $W^{s,p}(\Omega)$ is a Banach space with respect to the norm
\begin{align*}
\|u\|_{W^{s,p}(\Omega)}:=\left(\|u\|_{W^{m,p}(\Omega)}^p+\sum_{|\alpha|=m}\|D^\alpha u\|_{W^{\sigma,p}(\Omega)}^p\right)^{\frac 1p}.
\end{align*}

If $s=m$ is an integer, then $W^{s,p}(\Omega)$ coincides with the Sobolev space $W^{m,p}(\Omega)$.  Compare with \eqref{sob-emb1} we have the following general embedding.

\begin{theorem}\label{theo-sob-emb}
Let $\Omega\subset\RR^N$ be a bounded open set with Lipschitz continuous boundary. Then the following assertions hold.
\begin{enumerate}
\item  If $0<s\le r$ and $1<p\le q<\infty$ are real numbers such that $r-\frac Np= s-\frac Nq$, then $W^{r,p}(\RR^N)\hookrightarrow W^{s,q}(\RR^N)$.
 
 \item  If $0<s\le r$ and $1<p\le q<\infty$ are real numbers such that $r-\frac Np\ge s-\frac Nq$, then $W^{r,p}(\Omega)\hookrightarrow W^{s,q}(\Omega)$.
 \end{enumerate}
\end{theorem}
For more information on fractional order Sobolev spaces, we refer to \cite{AH,NPV,Gris,JW} and the references therein.

We also recall the following definition of the Besov space $B^{s}_{p,q}$, according to \cite[Chapter V, Section 5.1, Formula (60)]{STEIN}.
\begin{align}\label{besov-def}
	B^{s}_{p,q}(\RR^N) :=\left\{u\in L^p(\RR^N):\; \left(\int_{\RR^N}\frac{\norm{u(x+y)-u(y)}{L^p(\RR^N)}^q}{|y|^{N+qs}}\,dy\right)^{\frac{1}{q}}<\infty \right\}, \;\;\;1\le p,q\le\infty,\;\; 0<s<1.
\end{align}  

Notice that, when $p=q$, we have $B^{s}_{p,p}(\RR^N) = W^{s,p}(\RR^N)$. Finally, we recall the definition of the following potential space 
\begin{align}\label{sp-stein}
	\mathscr{L}^p_{2s}(\RR^N):=\Big\{u\in L^p(\RR^N):\;  \fl{s}{u}\in L^p(\RR^N)\Big\},\;\;\;1\le p\le\infty, \;\;s\ge 0,
\end{align}
introduced, for example, in \cite[Chapter V, Section 3.3, Formula (38)]{STEIN}. Note that this same space is sometimes denoted as $H^s_p(\RR^N)$ (see, e.g., \cite[Section 1.3.2]{TRIEB}). 
Here we  adopt the notation $\mathscr{L}^p_{2s}(\RR^N)$.

Finally, for the proof of our results, we will also need the following estimate. Let $A\subset\RR^N$ be a bounded set and $B\subset\RR^N$ an arbitrary set. Then there exists a constant $C>0$ (depending on $A$ and $B$) such that
\begin{align}\label{ine-dist}
|x-y|\ge C(1+|y|),\;\;\forall\;x\in A,\;\forall\;y\in \RR^N\setminus B,\;\mbox{dist}(A,\RR^N\setminus B)=\delta>0.
\end{align}
}

\textbf{Acknowledgments}:

\begin{itemize}
\item The work of Umberto Biccari was partially supported by the Advanced Grant DYCON (Dynamic Control) of the European Research Council Executive Agency, by the MTM2014-52347 Grant of the MINECO (Spain) and by the Air Force Office of Scientific Research under the Award No: FA9550-15-1-0027.

\item The work of Mahamadi Warma was partially supported by the Air Force Office of Scientific Research under the Award No: FA9550-15-1-0027. 

\item The work of Enrique Zuazua was partially supported by the Advanced Grant DYCON (Dynamic Control) of the European Research Council Executive Agency, FA9550-15-1-0027 of AFOSR, FA9550-14-1-0214 of the EOARD-AFOSR, the MTM2014-52347 Grant of the MINECO (Spain) and ICON of the French ANR.
\end{itemize}


\begin{thebibliography}{99}

\bibitem{AH} 
\newblock D.R.~Adams and L.I.~Hedberg.
\newblock{Function Spaces and Potential Theory}. 
\newblock Grundlehren der Mathematischen Wissenschaften \textbf{314}. Springer-Verlag, Berlin, 1996.

\bibitem{TURB}
\newblock O.G.~Bakunin. 
\newblock{Turbulence and diffusion: scaling versus equations}. 
\newblock Springer Science \& Business Media, 2008.

\bibitem{biccari}
\newblock U.~Biccari. 
\newblock\emph{Internal control for non-local Schr\"odinger and wave equations involving the fractional Laplace operator}. 
\newblock Preprint, {\tt arxiv.org/abs/1411.7800v2}.

\bibitem{fl_parabolic}
\newblock U.~Biccari, M.~Warma and E.~Zuazua,
\newblock\emph{Local regularity for fractional heat equations}.
\newblock Preprint, {\tt arxiv.org/abs/1704.07562}.

\bibitem{BTG}
\newblock M.~Bologna, C.~Tsallis and P.~Grigolini. 
\newblock\emph{Anomalous diffusion associated with non-linear fractional derivative Fokker-Planck-like equation: Exact time-dependent solutions}. 
\newblock Phys. Rev. \textbf{E 62} (2000), 2213--2218.

\bibitem{cozzi}
\newblock M.~Cozzi.
\newblock\emph{Interior regularity of solutions of non-local equations in Sobolev and Nikol'skii spaces}.
\newblock Ann. Mat. Pura Appl. \textbf{(2) 196} (2017), 555-578. 

\bibitem{NPV}
\newblock E.~Di Nezza, G.~Palatucci and E.~Valdinoci.
\newblock\emph{Hitchhiker's guide to the fractional Sobolev spaces}. 
\newblock Bull. Sci. Math. \textbf{136} (2012), 521--573. 

\bibitem{DPV} 
\newblock S.~Dipierro, G.~Palatucci and E.~Valdinoci.
\newblock\emph{Dislocation dynamics in crystals: a macroscopic theory in a fractional Laplace setting}. 
\newblock Commun. Math. Phys. \textbf{(2) 333} (2015), 1061--1105.

\bibitem{DZ} 
\newblock J.~Duoandikoetxea and E.~Zuazua.
\newblock\emph{Moments, masses de Dirac et d{\'e}composition de fonctions}. 
\newblock C. R. Acad. Sci. Paris, S{\'e}rie 1, Math{\'e}matique \textbf{(6) 315} (1992), 693--698.

\bibitem{EN}
\newblock  K-J. Engel and R. Nagel.
\newblock One-parameter Semigroups for Linear Evolution Equations. 
\newblock  Graduate Texts in Mathematics, \textbf{194}. Springer-Verlag, New York, 2000.

\bibitem{Val-D}
\newblock A. Fiscella, R. Servadei and E. Valdinoci.
\newblock\emph{Density properties for fractional Sobolev spaces}. 
\newblock Ann. Acad. Sci. Fenn. Math. \textbf{40} (2015), 235–-253. 

\bibitem{folland} 
\newblock G.B.~Folland.
\newblock{Real Analysis: Modern Techniques and their Applications}. 
\newblock John Wiley \& Sons, 2013.

\bibitem{GW-CPDE}
\newblock  C.G. Gal and M. Warma.
\newblock\emph{Nonlocal transmission problems with fractional diffusion and boundary conditions on non-smooth interfaces}.
\newblock Comm. Partial Differential Equations, to appear.

\bibitem{GO} 
\newblock G.~Gilboa and S.~Osher. 
\newblock\emph{Nonlocal operators with applications to image processing}.
\newblock Multiscale Model. Simul. \textbf{(3) 7} (2008), 1005--1028.

\bibitem{Gris} 
\newblock P.~Grisvard.
\newblock{Elliptic Problems in Nonsmooth Domains}. 
\newblock Monographs and Studies in Mathematics, \textbf{24}. Pitman, Boston, MA, 1985.

\bibitem{grubb} 
\newblock G.~Grubb.
\newblock\emph{Fractional Laplacians on domains, a development of Hörmander’s theory of $\mu$-transmission pseudodifferential operators}. 
\newblock Adv. Math. \textbf{268} (2015), 478-528.

\bibitem{JW} 
\newblock A.~Jonsson and H.~Wallin.
\newblock{Function Spaces on Subsets of $\mathbb R^N$}. 
\newblock Math. Rep. \textbf{2} (1984).

\bibitem{kato}
\newblock T.~Kato.
\newblock\emph{Strong $L^p$ solutions of the Navier-Stokes equation in $\RR^m$, with applications to weak solutions}.
\newblock Math.  Z. \textbf{(4) 187} (1984), 471--480.

\bibitem{LEVI} 
\newblock S.~Levendorski. 
\newblock\emph{Pricing of the American put under L\'evy processes}. 
\newblock Int. J. Theor. Appl. Finance \textbf{(03) 7} (2004), 303--335.

\bibitem{LPPS}
\newblock T.~Leonori, I.~Peral, A.~Primo and F.~Soria.
\newblock\emph{Basic estimates for solutions of a class of nonlocal elliptic and parabolic equations}. 
\newblock Discrete Contin. Dyn. Syst. \textbf{35} (2015), 6031--6068.

\bibitem{AN_DIFF}
\newblock M.M.~Meerschaert. 
\newblock\emph{Fractional calculus, anomalous diffusion, and probability}. 
\newblock Fractional dynamics, 265–-284, World Sci. Publ., Hackensack, NJ, 2012.

\bibitem{MOSER}
\newblock J.~Moser. 
\newblock\emph{A new proof of de Giorgi's theorem concerning the regularity problem for elliptic differential equations}. 
\newblock Comm. Pure Appl. Math. \textbf{13} (1960), 457--468.

\bibitem{PHAM} 
\newblock H.~Pham. 
\newblock\emph{Optimal stopping, free boundary, and American option in a jump-diffusion
model}. 
\newblock Appl. Math. Optim. \textbf{(2) 35} (1997), 145--164.

\bibitem{RS-PI}
\newblock X.~Ros-Oton and J.~Serra.
\newblock\emph{The Pohozaev identity for the fractional Laplacian}. 
\newblock Arch. Ration. Mech. Anal. \textbf{213} (2014), 587-–628. 

\bibitem{ROS}
\newblock X.~Ros-Oton and J.~Serra.
\newblock\emph{The Dirichlet problem for the fractional Laplacian: regularity up to the boundary}.
\newblock J. Math. Pures Appl. \textbf{(9) 101} (2014), 275--302.

\bibitem{RS-ES}
\newblock X.~Ros-Oton and J.~Serra.
\newblock\emph{The extremal solution for the fractional Laplacian}. 
\newblock Calc. Var. Partial Differential Equations \textbf{50} (2014), 723–-750. 

\bibitem{RS-FN}
\newblock X.~Ros-Oton and J.~Serra.
\newblock\emph{Boundary regularity for fully nonlinear integro-differential equations}.
\newblock Duke Math. J. \textbf{165} (2016), 2079–-2154. 

\bibitem{Val}
\newblock R. Servadei and E. Valdinoci.
\newblock\emph{On the spectrum of two different fractional operators}. 
\newblock Proc. Roy. Soc. Edinburgh \textbf{ Sect. A 144} (2014), 831–-855. 

\bibitem{STEIN}
\newblock E.~Stein 
\newblock\emph{Singular Integrals and Differentiability Properties of Functions}.
\newblock Princeton Mathematical Series, 1970.

\bibitem{stinga}
\newblock P.R.~Stinga 
\newblock\emph{Fractional powers of second order partial differential operators: extension problem and regularity theory}.
\newblock Ph. D Dissertation, Universidad Aut\'onoma de Madrid, Spain, 2010. 

\bibitem{tartar}
\newblock L.~Tartar.
\newblock\emph{An Introduction to Sobolev Spaces and Interpolation Spaces}.
\newblock Springer Science \& Business Media, 2007.

\bibitem{taylor}
\newblock M.E.~Taylor.
\newblock{Pseudodifferential Operators}.
\newblock Princeton Mathematical Series, Vol. 4, 1981.

\bibitem{TRIEB} 
\newblock H.~Triebel.
\newblock{Theory of Function Spaces II}. 
\newblock Monographs in Mathematics, \textbf{84}. Birkh\"auser Verlag, Basel, Boston, Berlin, 1992.

\bibitem{VAZ}
\newblock J.L.~V\'azquez.
\newblock\emph{Nonlinear diffusion with fractional Laplacian operators}.
\newblock Nonlinear partial differential equations, 271–-298, Proc. Abel Symp., \textbf{7}, Springer, Heidelberg, 2012.  
  
\bibitem{WAVE}
\newblock T.~Zhu and J.M.~Harris. 
\newblock\emph{Modeling acoustic wave propagation in heterogeneous attenuating media using decoupled fractional Laplacians}. 
\newblock Geophysics \textbf{(3) 79} (2014), T105--T116.

\end{thebibliography}
\end{document}